\documentclass[12pt]{amsart}

\usepackage{amsxtra,amssymb,amsmath,amscd,url,listings}
\usepackage{amsrefs}
\usepackage{enumitem}
\newlist{enumth}{enumerate}{1}
\setlist[enumth]{label=\emph{(\arabic*)}, ref=(\arabic*)}

\usepackage[utf8]{inputenc}
\usepackage{eucal}
\usepackage{fullpage}
\usepackage{scrtime}
\usepackage[colorlinks]{hyperref}
\usepackage{datetime}
\usepackage[all]{xy}
\usepackage{graphicx}
\usepackage{tikz-cd}
\usepackage{tikz-cd}
\tikzset{commutative diagrams/arrow style=math font}

\usepackage{mathrsfs}
\usepackage{xspace}

\usepackage{tensor}
\usepackage{braket}

\usepackage{todonotes}
\usepackage{tensor}

\shortdate

\DeclareMathSymbol{A}{\mathalpha}{operators}{`A}%
\DeclareMathSymbol{B}{\mathalpha}{operators}{`B}%
\DeclareMathSymbol{C}{\mathalpha}{operators}{`C}%
\DeclareMathSymbol{D}{\mathalpha}{operators}{`D}%
\DeclareMathSymbol{E}{\mathalpha}{operators}{`E}%
\DeclareMathSymbol{F}{\mathalpha}{operators}{`F}%
\DeclareMathSymbol{G}{\mathalpha}{operators}{`G}%
\DeclareMathSymbol{H}{\mathalpha}{operators}{`H}%
\DeclareMathSymbol{I}{\mathalpha}{operators}{`I}%
\DeclareMathSymbol{J}{\mathalpha}{operators}{`J}%
\DeclareMathSymbol{K}{\mathalpha}{operators}{`K}%
\DeclareMathSymbol{L}{\mathalpha}{operators}{`L}%
\DeclareMathSymbol{M}{\mathalpha}{operators}{`M}%
\DeclareMathSymbol{N}{\mathalpha}{operators}{`N}%
\DeclareMathSymbol{O}{\mathalpha}{operators}{`O}%
\DeclareMathSymbol{P}{\mathalpha}{operators}{`P}%
\DeclareMathSymbol{Q}{\mathalpha}{operators}{`Q}%
\DeclareMathSymbol{R}{\mathalpha}{operators}{`R}%
\DeclareMathSymbol{S}{\mathalpha}{operators}{`S}%
\DeclareMathSymbol{T}{\mathalpha}{operators}{`T}%
\DeclareMathSymbol{U}{\mathalpha}{operators}{`U}%
\DeclareMathSymbol{V}{\mathalpha}{operators}{`V}%
\DeclareMathSymbol{W}{\mathalpha}{operators}{`W}%
\DeclareMathSymbol{X}{\mathalpha}{operators}{`X}%
\DeclareMathSymbol{Y}{\mathalpha}{operators}{`Y}%
\DeclareMathSymbol{Z}{\mathalpha}{operators}{`Z}%

\renewcommand{\leq}{\leqslant}
\renewcommand{\geq}{\geqslant}
 
\numberwithin{equation}{section}

\newcommand{\uple}[1]{\text{\boldmath${#1}$}}

\newcommand{\Cc}{\mathbf{C}}

\newcommand{\Aa}{\mathbf{A}}

\newcommand{\Zz}{\mathbf{Z}}

\newcommand{\Rr}{\mathbf{R}}
\newcommand{\Gg}{\mathbf{G}}

\newcommand{\Hh}{\mathbf{H}}
\newcommand{\Qq}{\mathbf{Q}}

\newcommand{\Ff}{\mathbf{F}}

\newcommand{\Tt}{\mathbf{T}}

\newcommand{\mmu}{\boldsymbol{\mu}}

\newcommand{\mcO}{\mathscr{O}}

\newcommand{\mcT}{\mathscr{T}}

\newcommand{\Dir}[1]{\widehat{\Ff}^{\times}_{{#1}}}

\newcommand{\expect}{\mathbf{E}}

\def\loccit{loc.\kern3pt cit.{}\xspace}
\def\cf{see\kern.3em}
\def\Cf{See\kern.3em}
\def\eg{e.g.\kern.3em}

\def\resp{\text{resp.}\kern.3em}

\newcommand{\mods}[1]{\,(\mathrm{mod}\,{#1})}

\DeclareMathOperator{\rank}{rank}

\DeclareMathOperator{\vol}{Vol}

\DeclareMathOperator{\Reel}{Re}

\DeclareMathOperator{\Kl}{\mathrm{Kl}}

\DeclareMathOperator{\Hom}{Hom}

\DeclareMathOperator{\ft}{FT}

\newcommand{\eps}{\varepsilon}
\renewcommand{\rho}{\varrho}

\DeclareMathOperator{\GL}{\mathbf{GL}}

\newcommand{\demi}{{\textstyle{\frac{1}{2}}}}

\DeclareMathSymbol{\gena}{\mathord}{letters}{"3C}
\DeclareMathSymbol{\genb}{\mathord}{letters}{"3E}

\theoremstyle{plain}
\newtheorem{theorem}{Theorem}[section]
\newtheorem*{theorem*}{Theorem}

\newtheorem{proposition}[theorem]{Proposition}

\theoremstyle{remark}

\theoremstyle{definition}

\newtheorem{example}[theorem]{Example}
\newtheorem{remark}[theorem]{Remark}

\renewcommand{\geq}{\geqslant}
\renewcommand{\leq}{\leqslant}

\newcommand{\ov}[1]{\overline{#1}}

\newcommand{\Qlb}{{\ov{\Qq}_{\ell}}}

\begin{document}

\title{Toroidal families and averages of $L$-functions, I}
 
\author{\'Etienne Fouvry}
\address{Université Paris--Saclay,   CNRS \\
Laboratoire de Math\'ematiques d'Orsay\\
  91405 Orsay  \\France}
\email{etienne.fouvry@universite-paris-saclay.fr}

\author{Emmanuel Kowalski}
\address{ETH Z\"urich -- D-MATH\\
  R\"amistrasse 101\\
  CH-8092 Z\"urich\\
  Switzerland} \email{kowalski@math.ethz.ch}

\author{Philippe Michel}
\address{EPFL/SB/TAN, Station 8, CH-1015 Lausanne, Switzerland }
\email{philippe.michel@epfl.ch}

\subjclass[2010]{11M06,11T23,11A07}

\keywords{$L$-functions, toroidal family, Dirichlet characters,
  exponential sums over finite fields, toric congruences, moments of
  $L$-functions}

\begin{abstract}
  We initiate the study of certain families of $L$-functions attached to
  characters of subgroups of higher-rank tori, and of their average at
  the central point. In particular, we evaluate the average of the
  values $L(\demi,\chi^a)L(\demi,\chi^b)$ for arbitrary integers~$a$
  and~$b$ when $\chi$ varies over Dirichlet characters of a given prime
  modulus.
\end{abstract}
\date{\today}
\maketitle 


\begin{flushright}
  \textit{Dedicated to Henryk Iwaniec, avec respect, avec gratitude,
    avec admiration}
\end{flushright}

\section{Introduction}

The modern idea of ``family of $L$-functions'' in analytic number theory
crystallized in part in the work of Iwaniec and Sarnak~\cite{IS2} on
Landau--Siegel zeros (see also~\cite{IS1} for a general survey of
$L$-functions). A notable illustration of this idea is then found in
their paper~\cite{IS} concerning the mollified first and second moments
of central values of Dirichlet $L$-functions.

Our goal is to introduce a type of family related to Dirichlet
characters, which seems very natural, but has not been considered to the
best of our knowledge.

We call these families \emph{toroidal families}, as they arise from
families of Dirichlet (or automorphic) characters of \emph{algebraic
  tori} by restricting to those characters that belong to some
\emph{algebraic} subgroup of its group of characters. In the simplest
example, this means that we fix a non-zero vector $(a,b)\in\Zz^2$, and
for any modulus $q\geq 1$, we consider the group of pairs
$(\chi_1,\chi_2)$ of Dirichlet characters modulo $q$ such that
$\chi_1^a\chi_2^b=1$, and we study the properties as $q\to +\infty$ of
the $L$-functions $L(s,\chi_1)$ and $L(s,\chi_2)$ when $(\chi_1,\chi_2)$
satisfies this property. As we will see, even this first case leads to
interesting phenomena.

Our main result in this first paper is the asymptotic computation of the
average of the associated product $L$-values, in the case of prime
moduli.

\begin{theorem}[Second toroidal moment]\label{th-1}
  Let $a$ and $b$ be non-zero integers. There exists an absolute and
  effective constant~$c>0$ such that, defining
  $$
  \delta=\frac{c}{|a|+|b|},
  $$
  we have, for any prime number~$q$, the asymptotic formulas:
  \par
  \emph{(1)} If $a+b=0$, then
  $$
  \frac{1}{q-1}\sum_{\chi\mods{q}}L(\demi,\chi^a)L(\demi,\chi^{-a})=
  \log q+2C+O(q^{-\delta})
  $$
  where $C$ is a real number independent of~$a$.\footnote{\
    See~(\ref{eq-C}) for the value of~$C$.}
  \par
  \emph{(2)} If $a+b\not=0$, then
  $$
  \frac{1}{q-1}\sum_{\chi\mods{q}}L(\demi,\chi^a)L(\demi,\chi^b)=
  \alpha(a,b)+O(q^{-\delta}),
  $$
  where
  $$
  \alpha(a,b)=\begin{cases}
    \zeta(\frac{|a|+|b|}{2(a,b)})&\text{ if $ab<0$},\\
    1&\text{ if $ab>0$}.
  \end{cases}
  $$
\end{theorem}

Here is an interpretation of this result. First, if we take $b=-a$ with
$a\geq 2$ and suppose that we consider the primes $q\equiv 1\mods{a}$,
then the toroidal average of Theorem~\ref{th-1} can also be expressed as
the sum
\begin{equation}\label{eq-twist}
  \frac{1}{q-1}\sum_{\chi\mods{q}}|L(\demi,\chi^a)|^2=
  \sum_{\rho^a=1\mods{q}} \frac{1}{q-1} \sum_{\chi\mods{q}}
  \chi(\rho)|L(\demi,\chi)|^2.
\end{equation}

Our result is then an analogue of the estimation of the second moment
of critical values of Dirichlet $L$-functions twisted by $\chi(\rho)$,
which is the building block of the mollified or amplified second
moment (see, again,~\cite{IS}). The crucial difference from earlier
works is that in~(\ref{eq-twist}), the value of~$\rho$ is \emph{not} a
fixed integer, except if $\rho=1$ (which gives the main term) or if
$\rho=-1$, when $a$ is even (see also Remark~\ref{rm-pol}, where we
explain that one can also bound the further variant where~$\rho$ is a
root of an irreducible polynomial congruence of degree~$\geq 2$).

The fact that the second moment has the same order of magnitude as the
full average
$$
\frac{1}{q-1}\sum_{\chi\mods{q}}|L(\demi,\chi)|^2\sim \log q
$$
(the special case $a=-b=1$ of Theorem~\ref{th-1}, which was first proved
by Paley~\cite[Th.\,B]{paley})
indicates a kind of ``fair distribution'' of the large values of
$L(\demi,\chi)$ among the characters of the form $\chi^a$, or in other
words, among those that are trivial on $a$-th roots of unity
modulo~$q$.

Moreover, comparing with the formula for $a+b\not=0$, we conclude that
the sizes of the values $L(\demi,\chi^a)$ and $L(\demi,\chi^b)$ are
``uncorrelated'' if $a\not=b$. In this respect, the result is
reminiscent of the well-known fact that
$$
\sum_{n\leq x}d(n)d(n+1)
$$
is of order of magnitude $x(\log x)^2$ for $x$ large, whereas
$$
\sum_{n\leq x}d(n)^2 
$$
is of order of magnitude $x(\log x)^3$.

\begin{remark}
  (1) The first part of Theorem~\ref{th-1} implies (and is essentially
  equivalent) to the following: for a fixed integer~$a\geq 1$, as $q$
  varies among prime numbers, we have
  $$
  \frac{1}{q-1}\sum_{\substack{\chi\mods{q}\\\chi=\eta^a\text{ for
        some $\eta$}}} |L(\demi,\chi)|^2=
  \frac{1}{(a,q-1)}\Bigl(\log q+2C\Bigr)+O(q^{-\delta}),
  $$
  i.e., a second moment formula over the family of characters which are
  $a$-th powers (indeed, a character which is an $a$-th power is the
  $a$-th power of $(a,q-1)$ characters).

  (2) We have not attempted to obtain the best possible error terms
  available with current technology; our goal here is primarily to grasp
  the interesting new aspects of non-trivial toroidal families.

  (3) Suppose that $a+b\not=0$. In a standard way, Theorem \ref{th-1}
  establishes the existence of many characters modulo $q$ such that
  $$
  L(\demi,\chi^a)L(\demi,\chi^b)\not=0
  $$
  (a simultaneous non-vanishing result). Indeed, applying Cauchy's
  inequality twice and simply the upper bound
  $$
  \frac{1}{q-1}\sum_{\chi\mods q}|L(\demi,\chi^a)|^4
  \ll
  \frac{1}{q-1}\sum_{\chi\mods q}|L(\demi,\chi)|^4
  \ll (\log q)^4
  $$
  for $q$ prime (also proved first by Paley~\cite[Th.\,IV]{paley}),
  we see that
  $$
  |\{\chi\mods q\,\mid\, L(\demi,\chi^a)L(\demi,\chi^b)\not=0\}|\gg
  \frac{q}{(\log q)^4}.
  $$
  
  In fact, the power-saving error term in Theorem~\ref{th-1} means that
  it is straightforward to apply the mollification method (as
  in~\cite{IS}, see also~\cite{moments}) and prove that
  $$
  |\{\chi\mods q\,\mid\, L(\demi,\chi^a)L(\demi,\chi^b)\not=0\}|\gg q
  $$
  for $q$ prime, where the implicit constant depends on $(a,b)$ (see
  the work of Zacharias~\cite{zacharias} for some related results).

  (4) We note that although ``families of $L$-functions'' still appear
  most often in an informal manner based on concrete examples, there
  are significant attempts to provide a precise definition (e.g.,
  those of Sarnak, Shin and Templier~\cite{sst} and
  Kowalski~\cite{kowalski}).

  In an automorphic perspective, a natural framework seems to be the
  following. Let $K$ be a number field with ring of integers~$\mcO$
  and adèle ring~$\Aa_K$.  Let~$\Tt_1$ and~$\Tt_2$ be two tori (more
  generally, one can consider groups of multiplicative type) defined
  over~$\mcO$, and assume they are split for simplicity. Let
  $\varphi\colon \Tt_1\to \Tt_2$ be an algebraic morphism defined
  over~$\mcO$. Then $\chi\mapsto \chi\circ \varphi$ gives a
  map~$\varphi^*$ from characters of $\Tt_2(\Aa_K)/\Tt_2(K)$ to those
  of $\Tt_1(\Aa_K)/\Tt_1(K)$, i.e., from automorphic characters of one
  torus to the other.

  Let~$X_*(\Tt_1)$ be the abelian group of algebraic cocharacters
  of~$\Tt_1$. The $L$-group of $\Tt_1$ is the complex dual torus
  $\Hom(X_*(\Tt_1),\Cc^{\times})$ of $\Tt_1$, and similarly for $\Tt_2$
  (because we assumed that the tori are split).

  Let now $\rho\colon {}^L\Tt_1\to \GL_r$ be a faithful continuous
  representation. By the Langlands formalism, this defines for any
  automorphic character $\chi$ of $\Tt_1$ the associated $L$-function
  $L(s,\chi,\rho)$. Thus, varying $\chi$ over automorphic characters
  of~$\Tt_2$, we obtain a family of $L$-functions
  $L(s,\varphi^*\chi,\rho)$.  These are general forms of toroidal
  families.

  To illustrate this construction, let $K=\Qq$, $\Tt_1=\Gg_m^2$,
  $\Tt_2=\Gg_m$ and $\varphi(x,y)=x^ay^b$. Then
  ${}^L\Tt_1\simeq (\Cc^{\times})^2$ and we can take the embedding
  $\rho\colon {}^L\Tt_1\to \GL_2$ of ${}^L\Tt_1$ as diagonal
  matrices. For any $d\geq 1$, the automorphic characters of $\Gg_m^d$
  over~$\Qq$ are simply $d$-tuples of (primitive) Dirichlet
  characters; the map $\varphi^*$ then maps a primitive Dirichlet
  character~$\chi$ modulo~$q$ to the pair of Dirichlet characters
  $(\chi^a,\chi^b)$ (modulo primitivity issues), and we have
  $$
  L(s,\varphi^*\chi,\rho)=L(s,\chi^a)L(s,\chi^b),
  $$
  which recovers the type of $L$-functions in
  Theorem~\ref{th-1}. (Other options are possible, as pointed out by
  the referee; for instance taking $\rho\colon {}^L\Gg_m\to \GL_2$
  defined by $z\mapsto (z^a,z^b)$, we have also
  $L(s,\chi,\rho)=L(s,\chi^a)L(s,\chi^b)$.)

  We note that the set of automorphic characters of a torus with
  bounded conductor is amenable to analytic investigations in (more
  than) the generality we have described, due to the remarkable recent
  work of Petrow~\cite{petrow}. Thus some questions about toroidal
  families should be accessible in this generality.
\end{remark}

The outline of the remainder of this paper is the following: we will
first discuss some natural questions and generalizations of toroidal
averages, then consider separately some ingredients that should arise in
their evaluation in all cases, namely properties of certain exponential
sums and counting solutions of toric congruences, which (for later
purposes) we discuss in a broader context than needed in the present
paper (see Remark~\ref{rm-general}). We then give the proof of
Theorem~\ref{th-1}. After some elementary common considerations, this
splits naturally in two cases, depending on whether $a+b=0$ or not.

\subsection*{Notation}
\label{sec-conventions}

Given complex-valued functions $f$ and $g$ defined on a set $X$, we
write $f \ll g$ if there exists a real number $A \geq 0$ (called an
``implicit constant'') such that the inequality $|f(x)|\leq A g(x)$
holds for all $x \in X$. We write $f\asymp g$ if $f\ll g$ and~$g\ll f$.

We denote by~$|X|$ the cardinality of a set~$X$.


For any prime~$q$, we denote by $\Dir{q}$ the group of Dirichlet
characters modulo~$q$ (in other words, the group of characters
of~$(\Zz/q\Zz)^{\times}$).

For a Dirichlet character~$\chi$ modulo a prime number~$q$, we denote by
$$
\eps(\chi)=\frac{1}{\sqrt{q}}
\sum_{x\in\Ff_q^{\times}}\chi(x)e\Bigl(\frac{x}{q}\Bigr),
$$
the normalized Gauss sums, so that $|\eps(\chi)|=1$ if $\chi$ is
non-trivial. More generally, for an integer $k\geq 1$ and
a $k$-tuple $\uple{\chi}=(\chi_1,\ldots, \chi_k)$  of characters
modulo~$q$, we put
$$
\eps(\uple{\chi})=\prod_{j=1}^k\eps(\chi_j).
$$

Let~$K$ be a field. For an integer $k\geq 1$, we write
$x\cdot y=x_1y_1+\cdots+x_ky_k$ for the standard bilinear form on $K^k$.
For any integer~$d\geq 1$, we denote by $\mmu_d(K)$ the group of $d$-th
roots of unity which belong to~$K$.

Let~$k$ be an integer. A \emph{box} in $\Zz^k$ is a product
$B=I_1\times\cdots \times I_k$ where each $I_j=\{a_j,\ldots, b_j\}$ is
an interval in $\Zz$ with $a_j\leq b_j$. By the size of the box, we mean
simply its cardinality, and the boundary $\partial B$ is the union of
the sets
$$
I_1\times\cdots \times I_{j-1}\times \{a_j\}\times I_{j+1}\times \cdots
\times I_k,
\quad\quad
I_1\times\cdots \times I_{j-1}\times \{b_j\}\times I_{j+1}\times \cdots
\times I_k
$$
for $1\leq j\leq k$.

\subsection*{Acknowledgments}

This work was started during a visit of É.F. and Ph.M. to the
Forsch\-ung\-ssinstitut für Mathematik (FIM) at ETH Zürich, and
continued during a visit of É.F. and E.K. at the Bernoulli Center for
Fundamental Studies of EPF Lausanne. We all thank these institutions
for the excellent and stimulating work conditions.

We thank K. Soundararajan for pointing out to us the related works of
Pliego~\cite{pliego,pliego2} and of Khan--Milićević--Ngo~\cite{kmn},
and R. Nunes for pointing out the work of Bettin~\cite{bettin} (see
Section~\ref{sec-future} below). We also thank I. Shparlinski for
pointing out his work with M. Munsch~\cite{munsch-shparlinski}.

Ph. M. was partially supported by the SNF grant 200021\_197045.

We warmly thank the referee for a very detailed and careful report.

\section{Prospective and related works}\label{sec-future}

This section may be skipped in a first reading. We will list some
obvious variants and potential generalizations of toroidal families,
some of which we hope to study in later papers. See also the work of
Blomer, Fouvry, Milićević, Kowalski, Michel and Sawin~\cite{moments} for
a relatively systematic discussion of problems (and results) about
families of $L$-functions, many parts of which could be adapted to our
context.

\begin{enumerate}
\item A general form of toroidal averages is the following. Fix
  $k\geq 2$. Consider an integral matrix $ A=(a_{i,j})$ with $k$ columns
  (and an arbitrary finite number of rows), and for any prime number
  $q$, define
  $$
  \widehat{H}_{A}(q)=\{(\chi_1,\ldots,\chi_k)\mods{q}\,\mid\,
  \prod_{j=1}^k\chi_j^{a_{i,j}}=1\text{ for all }i \},
  $$
  a subgroup of the group of characters of $(\Ff_q^{\times})^k$.  We want
  to evaluate asymptotically
  $$
  \sum_{\uple{\chi}\in \widehat{H}_{A}(q)} \prod_{j=1}^k L(\demi,
  \chi_j)
  $$
  for $q$ prime. The case of Theorem~\ref{th-1} is essentially the
  case $k=2$ with $A=(-b,a)$. An analogue with $k=3$ would be
  the study of a sum like
  $$
  \sum_{\chi_1,\chi_2\mods{q}} L(\demi, \chi_1^a)L(\demi,\chi_2^b)L(\demi,
  \chi_1^c\chi_2^d).
  $$
  \par
  Besides the work in the present paper, we are only aware of two
  previous cases, which were considered by Bettin~\cite{bettin}, and
  Nordentoft~\cite{nordentoft} . Bettin considers the second moment
  for the subgroup defined by $\chi_1\cdots\chi_k\chi_{k+1}^{-1}=1$,
  and Nordentoft considers the subgroup defined by
  $\chi_1\cdots\chi_k=1$ (both for $k\geq 2$).  In addition, the case
  $a+b=0$ of Theorem~\ref{th-1} was considered after this paper was
  first submitted by Munsch and Shparlinski~\cite{munsch-shparlinski}.
  
\item In addition to the characters in some subgroup $\widehat{H}$ of
  tuples of Dirichlet characters modulo~$q$, it is natural to consider
  those in a \emph{coset} modulo such a subgroup. This means that in
  addition to the matrix~$A$, we also consider for any prime $q$ some
  characters $(\eta_1,\ldots,\eta_k)$ modulo~$q$ (arbitrarily chosen)
  and we look at the averages
  $$
  \sum_{\uple{\chi}\in \widehat{H}_{A}(q)} \prod_{j=1}^k L(\demi,
  \eta_j\chi_j).
  $$
  \par
  An example of this type with $k=3$ and $A=(1,1,1)$ has been considered
  by Zacharias~\cite{zacharias}, and Nordentoft~\cite{nordentoft} also
  considers such twists in his special case.

\item We focus here for simplicity on groups of characters which are
  dual groups of the groups of rational points on a torus which is
  split over~$\Ff_q$. It is likely that non-split cases would also
  lead to interesting questions. See \cite{MV} and \cite{nordentoft2}
  for works in that direction.
  
\item Similar questions can be raised for twisted $L$-functions of
  various kinds, for instance 
  \begin{equation}
    \label{second}
    \frac{1}{q-1}\sum_{\chi\mods{q}}L(\demi,f\otimes
    \chi^a)L(\demi,g\otimes\chi^b)
  \end{equation}
  for some modular forms~$f$ and~$g$, or the mixed moment
  \begin{equation}
    \label{mixed}
    \frac{1}{q-1}\sum_{\chi\mods{q}}L(\demi,\chi^a)L(\demi,f\otimes
    \chi^b).
  \end{equation}
  There are currently very few pairs $(a,b)$ for which these moments
  have been evaluated asymptotically; these are
  $$
  (a,b)=(1,-1),\ (2,-2)
  $$
  in the case of \eqref{second} (see \cite{KMSAnn,moments}), and
  $$
  (a,b)=(1,-1),(2,-2), (1,1),\ (2,2)
  $$
  in the case of \eqref{mixed} (see \cite{DK,zacharias}). Techniques
  inspired by \cite{Pisa} should make it possible to deal with more
  pairs $(a,b)$ for the mixed moment \eqref{mixed} (see Remark
  \ref{rempisa}).
  
  One can also of course study averages with derivatives of central
  values, and one might look at values $L(s,\chi)$ of the $L$-functions
  where $s\not=\demi$, or combine these variants.
  
\item A natural archimedean analogue of our problem consists in studying
  $$
  \int_{T}^{2T}\zeta(\demi+a_1it)\cdots \zeta(\demi+a_kit)dt
  $$
  for fixed $(a_1,\ldots,a_k)\in (\Rr^{\times})^k$. Some work has indeed
  been done by Pliego on this question in the case of three factors
  (see~\cite{pliego} and~\cite{pliego2}).  Some rather interesting
  phenomena related to diophantine approximation appear in his analysis.
  
\item Yet another question would be to study the \emph{distribution} of
  pairs
  $$
  (L(1,\chi^a),L(1,\chi^b))\in \Cc^2,
  $$
  when $\chi$ varies. As $q\to +\infty$, one will obtain a limiting
  distribution given by the law of the random vector
  $$
  \Bigl(\prod_p (1-U_p^a/p)^{-1},\prod_p (1-U_p^b/p)^{-1}\Bigr),
  $$
  where $(U_p)_p$ is a sequence, indexed by primes, of independent
  random variables all equidistributed on the unit circle. 
  Because of the dependency of the two components, it would be
  interesting to have, for instance, estimates for large values, large
  deviations, and related statistics, comparable to those known for
  $L(1,\chi)$ (see the paper~\cite{lamzouri} of Lamzouri).
  
  Similar properties hold for the values at any~$s$ with
  $\Reel(s)>1/2$. At~$s=1/2$, one can expect a Central Limit Theorem
  for
  $$
  \frac{ \log |L(\demi,\chi^a)L(\demi,\chi^b)|}{\sqrt{\log\log q}}
  $$
  (but note that it remains a major open problem to obtain such a
  result even for $\log|L(\demi,\chi)|$; Radziwi\l\l\ and
  Soundararajan~\cite{r-s} have however announced a proof of a result
  of this type when conditioning on characters with
  $L(\demi,\chi)\not=0$). Note that if we consider the pairs
  $(\log|L(\demi,\chi^a)|,\log|L(\demi,\chi^b)|)$, we can expect that
  after renormalization, the two components will become independent
  gaussians if $a\not=-b$ (formally, because
  $$
  \expect\Bigl(\sum_{p\leq X} \frac{U_p^a}{\sqrt{p}}\cdot
  \overline{\sum_{p\leq X}\frac{U_p^b}{\sqrt{p}}}\Bigr)=
  \sum_{p\leq X}\frac{1}{p}\expect(U_p^{a+b})=0
  $$
  by independance, which will imply that the limits should be
  uncorrelated gaussians).

\item We have considered prime moduli~$q$ for simplicity, but more
  general moduli can also be investigated. Moreover, other families of
  Dirichlet characters can be used to build toroidal averages, e.g., we
  can average as $Q\to+\infty$ over pairs $(\chi^a,\chi^b)$ as $\chi$
  ranges over all primitive Dirichlet characters with conductor
  $\leq Q$. This may open the door to higher values of~$k$, as happens
  with the moments of $L(\demi,\chi)$ (see for instance the
  work~\cite{ciw} of Conrey, Iwaniec and Soundararajan, and the recent
  work~\cite{chandee-li} of Chandee and Li).

  The paper~\cite{kmn} of Khan, Milićević and Ngo can be understood as
  somewhat similar: roughly speaking, they consider (mollified) first
  and second moments of the type
  $$
  \sum_{\chi\in\mathcal{O}}L(\demi,\chi),\quad\quad
  \sum_{\chi\in\mathcal{O}}|L(\demi,\chi)|^2
  $$
  where $\mathcal{O}$ is a (suitably large) Galois orbit of characters
  modulo~$p^k$ for a fixed prime number~$p$ and $k\to +\infty$ (in the
  simplest case, this amounts to characters of a fixed order, large
  enough that the orbit is big enough for averaging). The key new
  arithmetic ingredient they use is related to $p$-adic diophantine
  approximation~\cite[\S\,2.4]{kmn}.
  
  One can also look, from the point of view of changing moduli, at the
  toroidal family of primitive characters defined by the equation
  $\chi^2=1$, for $q\leq Q$, recovering this way the well-known family
  of real characters of bounded conductor as a kind of generalized
  toroidal family (although that is unlikely to be a useful point of
  view for the study of this family). We note that
  Chinta~\cite{chinta} deals with a somewhat related family, which
  corresponds to the equations
  $$
  \begin{cases}
    \chi_1^2=\chi_2^2=1\\
    \chi_1\chi_2\chi_3=1,
  \end{cases}
  $$
  but now (roughly speaking) with $\chi_1$, $\chi_2$ of moduli $q_1$,
  $q_2$ with~$q_1q_2\leq Q$.
  
\item We have worked with the base field~$\Qq$ for simplicity, but any
  other global field leads to similar questions. Function fields
  should, as usual, be somewhat more accessible, especially in the
  large finite field limit.
\end{enumerate}

\section{Exponential sums}\label{sec-exp-sums}

The first subsection can be mostly skipped in a first reading, since the
exponential sums relevant for the proof of Theorem~\ref{th-1} are
treated from scratch in Section~\ref{ssec-exp-th1}.

\subsection{General sums}\label{sec-exp-general}

We consider a more general situation for future reference.  Let
$k\geq 1$ be an integer. For any prime number~$q$ and any subgroup $H$
of~$(\Ff_q^{\times})^k$, we consider families of exponential sums of the
type
$$
u\mapsto \sum_{(x_j)\in uH}e\Bigl(\frac{x_1+\cdots+x_k}{q}\Bigr)
$$
for $u\in (\Ff_q^{\times})^k$. Special cases of these sums include the
hyper-kloosterman sums $\Kl_k$, which correspond to the subgroup~$H$
defined by $x_1\cdots x_k=1$: we then have
$$
\sum_{(x_j)\in uH}e\Bigl(\frac{x_1+\cdots+x_k}{q}\Bigr)= \sum_{x_1\cdots
  x_k=u_1\cdots u_k}e\Bigl(\frac{x_1+\cdots+x_k}{q}\Bigr).
$$

In fact, we observe that the condition $x\in uH$ is equivalent to
$x_j=u_jh_j$ for all~$j$ with $(h_j)\in H$, so that for
$u\in (\Ff_q^{\times})^k$, we can also write
$$
\sum_{(x_j)\in uH}e\Bigl(\frac{x_1+\cdots+x_k}{q}\Bigr)= \sum_{h\in
  H}e\Bigl(\frac{u_1h_1+\cdots+u_kh_k}{q}\Bigr).
$$

This formula identifies our family of exponential sums as the
restriction to $(\Ff_q^{\times})^k$ of the discrete Fourier transform of
the characteristic function of~$H\subset \Ff_q^k$. It is then most
convenient to extend the definition to all $u\in\Ff_q^k$ by the last
formula, as this allows us to use the standard formalism of Fourier
analysis.

Finally, it is convenient to normalize these sums before giving them a
name. There are various possibilities, and our choice is based on
obtaining generically sums of size one.  

More precisely, our interest lies in situations where $q$ varies
and~$H$ is for each~$q$ the group of $\Ff_q$-points of a fixed
algebraic subgroup~$\Hh$ of $\Gg_m^k$ over~$\Zz$. Concretely, this
means that we fix an integer matrix $A=(a_{i,j})$ with $k$ columns
(and a finite number of rows), and consider the subgroup
\begin{equation}\label{eq-subgp}
  \Hh_A(\Ff_q)=\{x=(x_j)\in (\Ff_q^{\times})^k
  \,\mid\, \prod_{j}x_j^{a_{i,j}}=1\text{ for all }
  i\}
\end{equation}
for all prime numbers~$q$. This is the group of $\Ff_q$-points of the
algebraic subgroup
$$
\Hh_A=\{x=(x_j)\in\Gg_m^k\,\mid\, \prod_{j}x_j^{a_{i,j}}=1\text{ for
  all } i\}
$$
of $\Gg_m^k$. For instance, taking $A=(1,\ldots,1)\in\Zz^k$, we
recover the subgroup which gives rise to $\Kl_k$.

We then define
\begin{equation*}
  T_A(u;q)=\frac{1}{q^{(k-\rank(A))/2}}\sum_{(x_j)\in
    \Hh_A(\Ff_q)}e\Bigl(\frac{x\cdot u}{q}\Bigr)
\end{equation*}
for $u\in \Ff_q^k$ (note that $\dim(\Hh)=k-\rank(A)$; in the case of
the hyper-Kloosterman sums, these sums are bounded for all
$u\in(\Ff_q^{\times})^{k}$ by Deligne's estimate).

Furthermore, we define
\begin{equation}\label{eq-perp}
  \Hh_A^{\perp}(\Ff_q)=\{\uple{\chi}=(\chi_j)\in (\Dir{q})^k\,\mid\,
  \uple{\chi}(x)=\prod_{j=1}^k\chi_j(x_j)=1\text{ for all } x=(x_j)\in
  \Hh_A(\Ff_q)\}.
\end{equation}

\begin{proposition}\label{pr-exp-gen}
  Let $A$ be an integral matrix as above.  For any prime number~$q$ and
  for any $u\in (\Ff_q^{\times})^k$, we have
  $$
  \frac{1}{|\Hh_A^{\perp}(\Ff_q)|} \sum_{\uple{\chi}\in
    \Hh_A^{\perp}(\Ff_q)} \eps(\uple{\chi})\overline{\uple{\chi}(u)}=
  \frac{1}{q^{\rank(A)/2}}T_A(u;q).
  $$
\end{proposition}

\begin{proof}
  This is a direct consequence of the definitions, and of orthogonality,
  in the form of the relation
  $$
  \frac{1}{|\Hh_A^{\perp}(\Ff_q)|}\sum_{\uple{\chi}\in
    \Hh_A^{\perp}(\Ff_q)}\uple{\chi}(xu^{-1})=
  \begin{cases}
    1&\text{ if } xu^{-1}\in \Hh_A(\Ff_q),\\
    0&\text{ otherwise,}
  \end{cases}
  $$
  which follows from the duality of finite abelian groups.
\end{proof}

We are interested in bounds for the exponential sums~$T_A(u;q)$. The
mean-square bound is elementary, but the best pointwise bounds lie
obviously much deeper, and in fact we only state these in the case of a
non-negative matrix, in the ``stratified form'' going back to Katz and
Laumon (see also~\cite{fouvry} for other applications of such
stratification results).

We require the following (non-standard) definition: a matrix
$A=(a_{i,j})$ as above is \emph{of affine type} if all coefficients
$a_{i,j}$ are non-negative, and if for any~$j$ with $1\leq j\leq k$,
there exists some $i$ such that~$a_{i,j}\geq 1$. The point is that for
such a matrix, we have
$$
\Hh_A=\{(x_j)\in\Aa^k\,\mid\, \prod_{j}x_j^{a_{i,j}}=1\text{ for all } i\},
$$
i.e., that all solutions of the equations with solutions in~$\Aa^k$
are in fact in~$\Gg_m^k$. (For instance, the empty matrix is not of
affine type.)

We also say that~$A$ is \emph{of connected type} if the subgroup~$X_A$
of~$\Zz^k$ generated by the rows of~$A$ is such that the quotient
abelian group $\Zz^k/X_A$ is torsion-free. This amounts to saying that
the character group of~$\Hh_A$ is torsion-free, and implies that the
group $\Hh_A$ (over~$\Zz$) is connected, i.e., is a torus (instead of
a general group of multiplicative type).

\begin{theorem}\label{th-fk}
  Let $A$ be an integral matrix as above.

  \emph{(1)} For any $q$ prime, we have
  $$
  \frac{1}{q^k}\sum_{u\in
    \Ff_q^k}|T_A(u;q)|^2=\frac{|\Hh_A(\Ff_q)|}{q^{k-\rank(A)}},
  $$
  and this quantity converges as $q\to +\infty$ to the number of
  connected components of~$\Hh_A$.

  \emph{(2)} Suppose that~$A$ is of affine and connected type. There
  exists a dense open set~$U\subset \Aa^k_{\Zz}$ such that
  $$
  |T_A(u;q)|\ll 1
  $$
  for all primes~$q$ and all~$u\in U(\Ff_q)$, where the implicit constant
  depends only on~$\rank(A)$.

  More precisely, there exist closed subschemes
  $X_k\subset \cdots\subset X_1\subset \Aa^k_{\Zz}$, defined by
  homogeneous equations, with~$X_j$ of relative dimension $\leq k-j$,
  such that for $q$ prime and $u\in (\Ff_q^{\times})^k$ with
  $u\notin X_j(\Ff_q)$, we have
  $$
  |T_A(u;q)|\ll q^{\max(0,(j-2)/2)}.
  $$
\end{theorem}

\begin{proof}
  (1) is a direct consequence of the discrete Plancherel formula and
  the fact that the cardinality of~$\Hh_A(\Ff_q)$ is
  $\sim q^{\dim(\Hh_A)}$ as $q\to +\infty$ if~$\Hh_A$ is connected (an
  elementary fact in our setting where we consider split tori).

  To prove~(2), note that the second statement implies the first with
  $U$ the complement of~$X_1$. To establish the latter, we need to
  invoke algebraic geometry, in the form of the result
  of~\cite[Th.\,1.2]{fouvry-katz} of Fouvry and Katz. We apply this to
  the algebraic variety~$\Hh_A$; since~$A$ is of affine type, the
  group~$\Hh_A$ is, as noted above, a closed subscheme
  of~$\Aa^k_{\Zz}$. It is an algebraic group of multiplicative type,
  and is connected since~$A$ is of connected type, so $\Hh_A(\Cc)$ is
  smooth and irreducible. Finally, it follows
  from~\cite[Th.\,8.1]{fouvry-katz} that the so-called $A$-number
  of~$\Hh_A$ is non-zero for~$q$ large enough (because~$\Hh_A$ is a
  torus, so that $|\Hh_A(\Ff_q)|$ is coprime to~$q$ for all~$q$). Thus
  all assumptions of the result of Fouvry and Katz are satisfied.
\end{proof}

\begin{remark}
  For future reference, we explain the origin of the exponential sums
  as trace functions in general (see~\cite{fkm-pisa} and~\cite{fkms}
  for general discussions of trace functions). We work over the finite
  field~$\Ff_q$ for a fixed prime~$q$. Pick a prime number
  $\ell\not=q$ and identify $\Qlb$ with~$\Cc$ by some fixed
  isomorphism. Let~$\psi$ be the additive $\Qlb$-valued character of
  $\Ff_q$ which is then identified with $x\mapsto e(x/q)$.

  The injective map $i\colon \Hh_A\to \Aa^k$ is a locally closed
  immersion (it is a closed immersion if all $a_{i,j}$ are
  non-negative). The formula defining $T_A(u;q)$ then shows that it
  coincides (up to a sign depending only on $A$) with the trace function
  over~$\Ff_q$ of Deligne's $\ell$-adic Fourier
  transform~$F_A=\ft_{\psi}(S_A)$, where
  $S_A=(i_!\Qlb)[\dim(\Hh_A)](\dim(\Hh_A))$.

  Since $i$ is an affine quasi-finite morphism, 
  the object~$S_A$ is a 
  perverse sheaf on~$\Aa^k$ (see~\cite[Cor.\,4.1.3]{bbdg}).  According
  to the formalism of Deligne's Fourier transform, this implies that
  $F_A$ is also a perverse sheaf 
  (see~\cite[Cor.\,2.1.5]{katz-laumon});

  Suppose now that $a_{i,j}\geq 0$ for all $i$ and~$j$, and that the gcd
  of the $a_{i,j}$ is one (as in the previous theorem). Then~$\Hh_A$ is
  (geometrically) connected, and the morphism~$i$ is a closed immersion,
  so the perverse sheaf~$S_A$ is in addition simple
  (see~\cite[Th.\,4.3.1]{bbdg}) and pure of weight~$0$ (because
  $i_*=i_!$ and because of Deligne's Riemann Hypothesis,
  see~\cite[5.1.14]{bbdg}), and so is~$F_A$
  (see~\cite[Th.\,2.2.1]{katz-laumon}); note that this means that we
  normalized the Fourier transform by a twist so that it preserves
  weights, in contrast with~\cite{katz-laumon}.

\end{remark}

Dually, we may construct using the matrix~$A$ the subgroups $X_A(\Ff_q)$
of $(\Dir{q})^k$ defined by the condition that
$\uple{\chi}\in X_A(\Ff_q)$ if
$$
\prod_{j=1}^k\chi_j^{a_{i,j}}=1
$$
for all~$i$. We can then form the subgroups
$$
X_A^{\perp}(\Ff_q)=\{x=(x_j)\in (\Ff_q^{\times})^k\,\mid\,
\uple{\chi}(x)=1\text{ for all }\uple{\chi}\in X_A(\Ff_q)\}.
$$

Elementary properties of the duality of finite abelian groups shows that
there exists a matrix~$B$ such that~$X_A^{\perp}(\Ff_q)=H_B(\Ff_q)$ for
all~$q$, and then~$H_B^{\perp}(\Ff_q)=X_A(\Ff_q)$ for all~$q$. Thus we
can pass back and forth between subgroups of $(\Ff_q^{\times})^k$ and
associated subgroups of~$(\Dir{q})^k$.

\begin{example}
  The example of Nordentoft~\cite{nordentoft} arises from the vector
  $A=(1,\ldots, 1)$ to construct $X_A(\Ff_q)$ which is the group of
  characters with $\chi_1\cdots \chi_k=1$. Then $X_A^{\perp}(\Ff_q)$ is
  the diagonal subgroup $\{(x,\ldots,x)\}\subset (\Ff_q^{\times})^k$ and
  corresponds to a matrix~$B$ with $\rank(B)=k-1$. The corresponding
  exponential sums are given by
  $$
  T_B(u;q)=\frac{1}{\sqrt{q}}\sum_{x\in\Ff_q^{\times}}
  e\Bigl(\frac{x(u_1+\cdots +u_k)}{q}\Bigr),
  $$
  and are therefore elementary (Ramanujan sums modulo~$q$).  Similarly,
  in the (slightly different) setting of of Khan, Milićević and Ngo,
  only elementary exponential sums arise. In the case of
  Bettin~\cite{bettin}, one can check that the relevant exponential sums
  are Kloosterman sums (for parameters of the form
  $(u_1+\cdots+u_k)u_{k+1}$).
\end{example}

In practice (this is the case in particular in Theorem~\ref{th-1}) we
often start with an integer~$d$ and an integral matrix $B=(b_{j,i})$
with $d$ columns and $k$ rows, and consider the subgroups of characters
$$
\widehat{H}(\Ff_q)\subset (\Dir{q})^k
$$
defined as the image of the map
$$
\varphi_B\colon (\chi_1,\ldots,\chi_d)\mapsto \Bigl(\prod_{i=1}^d
\chi_i^{b_{i,1}},\cdots, \prod_{i=1}^d \chi_i^{b_{i,k}}\Bigr).
$$

It is an elementary fact that there exists a (fixed) integral matrix~$A$
such that
$$
\widehat{H}(\Ff_q)=\widehat{H}_A(\Ff_q)
$$
for all $q$.  Note that the map $\varphi_B$ is not necessarily
injective, so that summing over $(\chi_1,\ldots,\chi_d)$ is not
equivalent to summing over the image of~$\varphi_B$. However, since the
number of pre-images of a character $\eta\in \varphi_B((\Dir{q})^d)$ is
independent of~$\eta$, this is only a cosmetic issue.


\subsection{The case of Theorem~\ref{th-1}}\label{ssec-exp-th1}

We take $k=2$ and $A=(a,b)\in\Zz^2$, $ab\not=0$. The group~$\Hh_A$ is
defined by the equation $x^ay^b=1$, and it has dimension~$1$; it is
connected if and only if $a$ and~$b$ are coprime.

For any prime number~$q$, we define the exponential sums
$T_{a,b}(u,v;q)$ by
$$
T_{a,b}(u,v;q)=\frac{1}{\sqrt{q}} \sum_{\substack{x,y \in
    \Ff_q^{\times}\\x^ay^b=1}}e\Bigl(\frac{ux+vy}{q}\Bigr),
$$
for $u$, $v\in\Ff_q$, and we define
$$
\widetilde{T}_{a,b}(u;q)= \frac{1}{\sqrt{q}} \sum_{\substack{x,y\in
    \Ff_q^{\times}\\x^ay^b=u}}e\Bigl(\frac{x+y}{q}\Bigr),
$$
for $u\in\Ff_q^{\times}$, so that
$T_{a,b}(u,v;q)=\widetilde{T}_{a,b}(u^av^b;q)$ for $u$ and $v$ in
$\Ff_q^{\times}$.

The relation with Gauss sums can be expressed here by the formula
\begin{equation}\label{eq-tab-gauss}
  \widetilde{T}_{a,b}(u;q)=\frac{\sqrt
    q}{q-1}\sum_{\chi\mods{q}}\eps(\chi^a)\eps(\chi^b) \overline{\chi(u)}
\end{equation}
for all~$u\in\Ff_q^{\times}$.

We note in passing that if $q\geq 3$, then we also have
\begin{equation}\label{eq-ab2a2b}
  \widetilde{T}_{a,b}(u;q)+\widetilde{T}_{a,b}(-u;q)=\widetilde{T}_{2a,2b}(u^2;q)
\end{equation}
for $u\in\Ff_q^{\times}$.

The algebraic structure of these exponential sums is different
depending on whether $a+b=0$ or not. In the former case, we have an
elementary result.

\begin{proposition}\label{pr-tama}
  Suppose that $a+b=0$. For any $u\in\Ff_q^{\times}$, we have
  $$
  \widetilde{T}_{a,b}(u;q)=\alpha_1+\alpha_2
  $$
  where
  \begin{gather*}
    \alpha_1=
    \begin{cases}
      \sqrt{q}&\text{ if } u=(-1)^a,\\
      0&\text{ otherwise},
    \end{cases}
    \\
    \alpha_2=\begin{cases}
      -\frac{(a,q-1)}{\sqrt{q}}&\text{ if
                                 $u\in(\Ff_q^{\times})^a$},\\
      0&\text{ otherwise}.
    \end{cases}
  \end{gather*}
\end{proposition}

\begin{proof}
  It is simpler here to use the formula~(\ref{eq-tab-gauss}) and the
  fact that $\eps(\chi^a)\eps(\chi^{-a})$ is equal to~$\chi^a(-1)$ if
  $\chi^a\not=1$ and to $1/q$ if~$\chi^a=1$. Thus
  $$
  \widetilde{T}_{a,b}(u;q)=\frac{\sqrt{q}}{q-1}\Bigl(
  \sum_{\chi}\chi^a(-1)\overline{\chi(u)}-
  \sum_{\chi^a=1}\chi^a(-1)\overline{\chi(u)}
  +\frac{1}{q}\sum_{\chi^a=1}\overline{\chi(u)}
  \Bigr).
  $$

  The first (resp. second and third) terms in this expression give the
  corresponding values in the statement by orthogonality of characters
  of~$\Ff_q^{\times}$ and of~$\Ff_q^{\times}/(\Ff_q^{\times})^a$.
\end{proof}

\begin{remark}
  From the point of view of trace functions, this case is actually
  rather delicate, since one should see $\widetilde{T}_{a,-a}(u;q)$ as
  the trace function of a perverse object, and not of a single sheaf,
  and this object is neither simple nor pure
  (see~\cite[Lemma\,8.4.8]{katz-esde} for what amounts to the case
  $a=-1$).
\end{remark}

The case $a+b\not=0$ gives rise to more interesting exponential
sums. We can view them as trace functions again, but the main property
that we will use also has a more elementary proof.

\begin{proposition}\label{pr-tab-weil}
  Suppose that $a+b\not=0$. We have
  $$
  |\widetilde{T}_{a,b}(u;q)|\leq |a|+|b|
  $$
  for all primes~$q\geq \max(|a|,|b|)^2$ and
  all~$u\in\Ff_q^{\times}$. Additionally, the bound above can be
  improved to $\max(|a|,|b|)$ if $a$ and $b$ have different signs.
\end{proposition}

\begin{proof}
  Let $\delta=(a,b)$ and pick integers $\alpha$, $\beta$ such that
  $$
  \beta a+\alpha b=\delta.
  $$
  \par
  It is then straightforward to check that the map
  $$
  \Ff_q^{\times}\times \mmu_{\delta}(\Ff_q)\to
  \{(x,y)\in(\Ff_q^{\times})^2\,\mid\, x^ay^b=1\}
  $$
  defined by
  $$
  (t,\rho)\mapsto (\rho^{\beta}t^{b/\delta},\rho^{\alpha} t^{-a/\delta})
  $$
  is a bijection with inverse
  $$
  (x,y)\mapsto (x^{\alpha}y^{-\beta},x^{a/\delta}y^{b/\delta}).
  $$
  \par
  Let now~$u\in\Ff_q^{\times}$. If $u$ is not of the form $x_0^ay_0^b$
  with $(x_0,y_0)\in(\Ff_q^{\times})^2$, then $\widetilde{T}_{a,b}(u;q)=0$.
  Otherwise, fix such representation $u=x_0^ay_0^b$. Then the above
  allows us to write
  $$
  \widetilde{T}_{a,b}(u;q)= \sum_{\rho^{\delta}=1}
  \frac{1}{\sqrt{q}}\sum_{t\in\Ff_q^{\times}}
  e\Bigl(\frac{x_0\rho^{\beta}t^{b/\delta}+y_0\rho^{\alpha}t^{-a/\delta}}{q}\Bigr).
  $$
  \par
  For each $\rho\in\Ff_q^{\times}$ such that $\rho^{\delta}=1$, the
  inner sum over~$t$ is a standard Weil sum with a rational
  function. The corresponding bounds are as follows (see,
  e.g.,~\cite[(3.5.2)]{deligne}), where we assume that $|a|/\delta<q$
  and $|b|/\delta<q$.

  \textbf{Case 1.} If $a$ and $b$ have the same sign, then the rational
  function has poles at~$0$ and~$\infty$, and the Weil bound is
  $$
  |\widetilde{T}_{a,b}(u;q)|\leq \frac{|a|+|b|}{\delta}.
  $$
  In fact, the sum is ``pure'' in that case, i.e., it is the sum of as
  many complex numbers of modulus~$1$ as the right-hand upper-bound.

  \par
  \textbf{Case 2.} If~$a$ is negative and~$b$ is positive, then the
  function
  $t\mapsto \rho^{\beta}t^{b/\delta}+\rho^{\alpha}t^{-a/\delta}$ is a
  \emph{polynomial} of degree $\max(|a|/\delta, b/\delta)$ (the assumption
  $a+b\not=0$ ensures there is no cancellation), so the Weil bound is
  $$
  \Bigl|\frac{1}{\sqrt{q}}\sum_{t\in\Ff_q}
  e\Bigl(\frac{\rho^{\beta}t^{b/\delta}+\rho^{\alpha}t^{-a/\delta}}{q}\Bigr)
  \Bigr|\leq \frac{\max(-a,b)}{\delta}-1,
  $$
  (and this sum is ``pure'') where the sum ranges over
  all~$t\in\Ff_q$. Thus, we get a bound
  $$
  |\widetilde{T}_{a,b}(u;q)|\leq
  \max(|a|-1,b-1)+\frac{\delta}{q^{1/2}}\leq
  \max(|a|,|b|)
  $$
  provided $q>\max(|a|,|b|)^2$.
  \par
  \textbf{Case 3.} If $a$ is positive and $b$ is negative, then the
  symmetry $\widetilde{T}_{-a,-b}(u;q)=\widetilde{T}_{a,b}(u^{-1};q)$,
  reduces the question to Case~2.
\end{proof}

\begin{remark}
  (1) This result is also proved by Pierce~\cite[\S\,3.1]{Pierce}, in
  the case $(a,b)=1$.
  \par
  (2) Elaborating this argument using the formalism of $\ell$-adic
  sheaves and trace functions, one can show that for any prime
  number~$q>\max(|a|,|b|)^2$, there exists a lisse sheaf~$\mcT_{a,b}$
  on~$\Gg_m$ over~$\Ff_q$ such that:
  \begin{itemize}
  \item The trace function of~$\mcT_{a,b}$ coincides with
    $\widetilde{T}(u;q)$;
  \item We have
    $$
    \rank(\mcT_{a,b})=\begin{cases}
      |a|+|b|&\text{ if } ab\geq 1\\
      \max(|a|,|b|)&\text{ if } ab\leq -1,
    \end{cases}
    $$
    and moreover the conductor (in the sense of~\cite{fkm1})
    of~$\mcT_{a,b}$ is $\ll 1$ for all~$q$;
  \item The sheaf~$\mcT_{a,b}$ is mixed of weights $\leq 0$, so that
    $$
    |\widetilde{T}(u;q)|\leq |a|+|b|
    $$
    for all~$u\in\Ff_q^{\times}$;
  \item If $a$ and $b$ have the same sign, then the sheaf~$\mcT_{a,b}$ is
    pure of weight~$0$. It is geometrically irreducible if and only if
    $(a,b)=1$.
  \end{itemize}

  This means in particular that we can apply to these exponential sums
  results such as those in~\cite{fkm1, fkm2}.
\end{remark}

\section{Counting solutions of toric congruences}\label{sec-diophantine}

\subsection{A general problem}

We consider first the general setting, as in
Section~\ref{sec-exp-general}, from which we borrow some notation.

Fix an integer~$k\geq 1$. The basic question is now, given a prime $q$
and a subgroup $H\subset (\Ff_q^{\times})^k$, to estimate the number
of elements of some non-empty box $B=I_1\times \cdots \times I_k$ in
$\Zz^k$, where $I_j$ is an interval, whose reduction modulo~$q$ lies
in~$H$.

More generally, it is natural to count solutions in cosets, since this
provides some potential flexibility in arguing by induction. Thus,
given an integral matrix~$A$ with $k$ columns (and an arbitrary finite
number of rows) and the subgroups~$\Hh_A(\Ff_q)$ defined
in~(\ref{eq-subgp}), we define
$$
M_A(u,B;q)=|\{x\in B\,\mid\, x\mods{q} \in u\Hh_A(\Ff_q)\}|
$$
for $u\in(\Ff_q^{\times})^k$, and we write simply $M_A(B;q)=M_A(1,B;q)$.

Concretely, this amounts to counting the solutions to a system of
congruences of the form
\begin{equation}\label{eq-congruence}
  x_1^{a_{i,1}}\cdots x_k^{a_{i,k}}\equiv y_i\mods{q}
\end{equation}
for all~$i$, with $(x_j)\in B$ and $y_i\in\Ff_q^{\times}$. We call such
systems ``toric congruences''.

Here also, some normalization plays an important role in applications to
toroidal families, and we define for this purpose
\begin{equation}\label{eq-congruences-norm}
  N_A(u,B;q)=\frac{M_A(u,B;q)}{\sqrt{|B|}},\quad\quad
  N_A(B;q)=\frac{M_A(1,B;q)}{\sqrt{|B|}}.
\end{equation}

One can expect the problem of estimating these quantities to be quite
difficult when the intervals $I_j$ have small size, whereas the
Riemann Hypothesis over finite fields can be applied to get an
asymptotic formula once they are big enough. Here is a sample result
of the latter kind, where we recall that matrices of affine and
connected types where defined before Theorem~\ref{th-fk}.

\begin{proposition}
  Assume that $A$ is of affine and connected type. Assume moreover
  that $\Hh_A(\Cc)$ is not contained in any of the sets defined by
  $x_i=x_j$ for some integers $i\not=j$, or $x_j=1$ for some
  integer~$j$.

  For any $\eps>0$, for any prime number $q$ and for any box of the type
  $B=\{1,\ldots,x\}^k$ with $1\leq x<q$, we have
  $$
  M_A(B;q)=\frac{|B|}{q^{\rank(A)}}+O\Bigl(q^{(k-\rank(A))/2+\eps}\Bigl(1+
  \frac{1}{q^{1/2}} \Bigl(\frac{|B|}{q^{1/2}}\Bigr)^{k-\rank(A)}
  \Bigr)\Bigr)
  $$
  where the implicit constant depends only on~$A$ and~$\eps$.
\end{proposition}

Although the result is valid uniformly, it is of course only non-trivial
when $|B|$ is large enough.

\begin{proof}
  This is a direct application of the general
  result~\cite[Cor.\,1.5]{fouvry-katz} of Fouvry and Katz. We check the
  assumptions of this result:
  \par
  (1) since the matrix~$A$ is non-negative, the algebraic torus~$\Hh_A$
  is a closed subvariety of~$\Aa^k$; since the gcd of the coefficients
  is one, the variety $\Hh_A(\Cc)$ is connected;
  \par
  (2) by~\cite[Th.\,8.1]{fouvry-katz}, the so-called ``$A$-number''
  of~$\Hh_A$ over~$\Ff_q$ is non-zero for all primes~$q$ large enough;
  \par
  (3) finally, $\Hh_A(\Cc)$ is not contained in any affine hyperplane
  in~$\Cc^k$; indeed, if this were so, it would follow that the $k+1$
  characters
  $$
  x\mapsto x_j,\quad 1\leq j\leq k, \quad\quad x\mapsto 1
  $$
  of~$\Hh_A(\Cc)$ are linearly dependent. Since the characters of any
  group are linearly independent, this is only possible if two at least
  of these characters coincide, which contradicts our assumption.
\end{proof}

We will use another bound in the proof of Theorem~\ref{th-1}, which is
based on the geometry of numbers and applies to linear toric
congruences, i.e., congruences of the form
$x_i\equiv \rho x_j\mods{q}$ for some~$\rho\in\Ff_q^{\times}$. This
corresponds to
$$
A=(0,\ldots,0,1,0,\ldots,0,-1,0,\ldots,0),
$$
with a $1$ at the $i$-th position and a $-1$ at the $j$-th position.

\begin{proposition}\label{pr-geometry}
  Let~$k\geq 1$ be an integer. Let~$q$ be a prime number. Fix integers
  $i$, $j$ with $1\leq i\not=j\leq k$.  Let $u\in
  (\Ff_q^{\times})^k$. Let $\lambda_1$ be the minimum of the lattice
  $$
  \Lambda=\{y\in\Zz^k\,\mid\, u_iy_i\equiv u_jy_j\mods{q}\}\subset \Zz^k.
  $$

  For any box $B\subset \Zz^k$, we have
  $$
  |\{x\in B\,\mid\, u_ix_i\equiv u_j x_j \mods{q}\}| \ll
  \frac{|B|}{q}+\Bigl(\frac{\Delta}{\lambda_1}+1\Bigr)^{k-1},
  $$
  where $\Delta$ is the size of~$\partial B$ and the implicit constant
  depends only on~$k$.
\end{proposition}

\begin{proof}
  Write
  $$
  B=\prod_{j=1}^k\{a_j,a_j+1,\ldots,b_j\}\subset \Zz^k,
  $$
  and let
  $$
  B^0=\prod_{j=1}^k[a_j,b_j]\subset \Rr^k.
  $$

  We are therefore computing the size of $\Lambda\cap B^0$.  The set
  $B^0$ is convex and bounded, and its volume is $\ll |B|$.  The
  result thus follows from the estimate
  $$
  |\Lambda\cap B^0|\ll
  \frac{\vol(B^0)}{q}+\Bigl(\frac{\Delta}{\lambda_1}+1\Bigr)^{k-1},
  $$
  where the implicit constant depends only on~$k$, which is a case of
  the so-called Lipschitz Principle in the geometry of numbers. Among
  results of this type which imply this estimate, we appeal to one of
  Widmer~\cite[Cor.\,5.3]{widmer} (which is much more general).

  More precisely, the determinant of~$\Lambda$ is~$q$ under our
  assumptions on~$u$.  Denoting by~$\Omega$ the orthogonality defect
  of~$\Lambda$ (see~\cite[p.\,4805]{widmer}), and applying loc. cit., we
  get
  $$
  \Bigl|\,|\Lambda\cap B^0|-\frac{\vol(B^0)}{q}\Bigr| \leq 3^k\,(2k)\,
  \Bigl(\frac{\sqrt{k}\Omega \Delta}{\lambda_1}+1\Bigr)^{k-1},
  $$
  since the boundary of the box $B^0$ is
  the union of $2k$ faces, each of which is the image of a Lipschitz map
  $[0,1]^{k-1}\to \partial B^0$ with Lipschitz constant bounded by the
  constant~$\Delta$.  The bound $\Omega\ll 1$ holds
  by~\cite[Lemma\,4.4]{widmer}, where the implicit constant depends only
  on~$k$, and the proposition follows.
\end{proof}

Finally, we conclude with some speculations, maybe rather naive, but
which we hope can provide definite targets for investigations of general
toric congruences.

We consider a box $B$ where the sides $I_j$ are of the form
$I_j=\{1,\ldots, A_j\}$ for some positive integers $ A_j<q$. There are
some ``systematic'' sources of solutions of the toric
congruences~(\ref{eq-congruence}). These are provided by the non-trivial
morphisms
$$
\varphi\colon \Ff_q^{\times}\to (\Ff_q^{\times})^k
$$
of the form
$$
\varphi(x)=(x^{b_1},\ldots,x^{b_k})
$$
with $b_i\geq 0$, not all zero, whose image lies in~$\Hh_A$. This means
that $(b_i)$ is a non-zero and non-negative solution of the system of
linear equations
\begin{equation*}
  b_1a_{i,1}+\cdots +b_ka_{i,k}=0
\end{equation*}
for all~$i$. These homomorphisms correspond to algebraic group
morphisms $\Gg_m\to \Hh_A$, and only depend on the matrix~$A$.

Given such a morphism $\varphi$, we have
$\varphi(x)\in B\cap \Hh_A(\Ff_q)$ provided
$$
1\leq x\leq A_j^{1/b_j}\quad \text{ for all $j$ such that $b_j\geq 1$},
$$
where we use in an essential way the condition $b_j\geq 0$.

The kernel of the morphism $\varphi$ if of size $\ll 1$ for all $q$ (it
is the group of $\gcd(b_1,\ldots,b_k)$-th roots of unity
in~$\Ff_q^{\times}$). The morphism $\varphi$ therefore provides
$$
\gg \min_{b_j\geq 1}(A_j^{1/b_j})
$$
solutions to the system of congruences~(\ref{eq-congruence}) with
right-hand side $1$.

We now assume that $A_j\geq q^{\alpha_j}$ for some $\alpha_j>0$, so
that the systematic solutions associated to a given $\varphi$ are
increasingly numerous. We then let $M_A^{\rm syst}(B;q)$ denote the
number of elements of $B\cap \Hh_A(\Ff_q)$ which belong to the image of
a positive morphism (some solutions of the congruences might of course
lie in the image of more than one morphism). One may then speculate
that
$$
M_A(B;q)=M_A^{\rm syst}(B;q)+O\Bigl(q^{\eps}\Bigl(1+
\frac{|B|}{q^{k-\rank(A)}}\Bigr)\Bigr)
$$
for any $\eps>0$ (the quantity $|B|/q^{k-\rank(A)}$ on the right-hand
side is of probabilistic nature: the probability for a uniformly random
element of $(\Ff_q^{\times})^k$ to belong to $\Hh_A(\Ff_q)$ is about
$q^{-\dim(\Hh_A)}=q^{-(k-\rank(A))}$).

\begin{remark}
  (1) There is a vague analogy here with the Pila--Wilkie counting
  theorem and the distinction between algebraic and transcendental parts
  of definable sets in o-minimal structures (see, e.g., the recent
  survey of Bhardwaj and van den Dries~\cite{bvdd} for a discussion of
  the Pila--Wilkie theorem and its applications).
  \par
  (2) If the matrix $A$ is non-negative, so that $a_{i,j}\geq 0$, then
  we have $M_A^{\rm syst}(B;q)=0$.
\end{remark}

\subsection{The case of Theorem~\ref{th-1}}\label{ssec-congruences}

We now discuss the example which is relevant for
Theorem~\ref{th-1}.

We have now $k=2$ and $A=(a,b)\in\Zz^2$ with $ab\not=0$.  For~$q$ prime,
the subgroup $\Hh_A(\Ff_q)$ is the subgroup of $(\Ff_q^{\times})^2$ given
by the equation $x^ay^b=1$.

A trivial bound is
$$
M_{a,b}(u,I\times J;q)\ll |I|\Bigl(\frac{|J|}{q}+1\Bigr), \quad\quad
M_{a,b}(u,I\times J;q)\ll |J|\Bigl(\frac{|I|}{q}+1\Bigr),
$$
(indeed, for the first, note that for each~$m\in I$, the possible~$n$
such that $m^an^b\equiv 1\mods{q}$ lie in the intersection of~$J$ with
at most~$b$ arithmetic progressions modulo~$q$). In particular
\begin{equation}\label{eq-trivial}
  N_{a,b}(u,I\times J;q)\ll \frac{\sqrt{|I|}}{\sqrt{|J|}}
  \Bigl(\frac{|J|}{q}+1\Bigr).
\end{equation}

The following is rather deeper, and is related to the methods used to
prove the Burgess bound.

\begin{theorem}[Pierce]\label{thmpierce}
  Suppose that $b\not=0$ and that $a/b$ is not a negative integer.  Let
  $k\geq 1$ be an integer. If $M$ and~$N$ are positive integers $<q$
  such that
  \begin{equation*}
    M\leq \demi q^{\frac{k+1}{2k}},\ N\leq \tfrac{q}{4},
  \end{equation*}
  then we have
  $$
  M_{a,b}([1,M]\times[1,N];q)\ll M^{\frac{k}{k+1}}N^{\frac{1}{2k}}(\log
  q)^{1/(2k)}
  $$
  for all primes~$q$, where the implicit constant depends only
  on~$a$, $b$ and~$k$.
\end{theorem}

\begin{proof}
  This is a special case of~\cite[Th.\,4]{Pierce} for~$q$ prime,
  noting that Pierce's equations are written in the form $x^a=y^b$,
  hence our condition on $(a,b)$ is the opposite of that in loc. cit.;
  the other condition $(b,q)=1$ in loc. cit.  is true for all large
  enough~$q$.
\end{proof}


On the other hand, Proposition~\ref{pr-geometry} implies the
following result.

\begin{proposition}\label{pr-lattice2}
  Let~$q$ be a prime number. For any $\alpha\in\Ff_q^{\times}$, and any
  box~$B=I\times J\subset \Zz^2$, we have
  $$
  M_{1,-1}((\alpha,1),I\times J;q)\ll \frac{|B|}{q}+\frac{|I|+|J|}{\lambda}+1
  $$
  where $\lambda$ is the minimum of the lattice
  $$
  \Lambda=\{(m,n)\in\Zz^2\,\mid\, m\equiv \alpha n\mods{q}\}.
  $$
\end{proposition}

\begin{remark}
  (1) The motivation in the paper of Pierce is the study of class
  groups of imaginary quadratic fields. In this respect, we may also
  mention another work of Heath--Brown and
  Pierce~\cite[\S\,4]{hb-pierce}, where toric congruences also appear
  and are studied in part through lattice-point counting methods.

  (2) Special cases of these toric congruences appeared in earlier
  work of Heath--Brown~\cite{hb} concerning squarefree numbers in
  arithmetic progressions, and more recently in related work of
  Nunes~\cite{nunes}.
\end{remark}

\section{Beginning of the proof}

We begin the proof of Theorem~\ref{th-1} with some general
preliminaries. For any Dirichlet character $\chi$ modulo~$q$, we denote
$t(\chi)=(1-\chi(-1))/2$.

Let~$(a,b)\in\Zz^2$ be integers with $ab\not=0$.  Let~$q$ be a prime
number. We intend to compute
$$
M_{a,b}(q)=
\frac{1}{q-1}\sum_{\chi\mods{q}}L(\demi,\chi^a)L(\demi,\chi^b).
$$

For any Dirichlet character~$\chi$ modulo~$q$, we have
$$
L(s,\chi^a)L(s,\chi^b)=\sum_{m,n\geq
  1}\frac{\chi(m^an^b)}{(mn)^s},\quad
\quad
\Reel(s)>1.
$$
\par
Suppose that~$\chi^a\not=1$ and~$\chi^b\not=1$, so that the
$L$-functions $L(s,\chi^a)$ and $L(s,\chi^b)$ have functional equations
with signs
$$
i^{-t(\chi^a)}\eps(\chi^a),\quad\quad i^{-t(\chi^b)}\eps(\chi^b),
$$
respectively.  For any positive real numbers~$X$ and~$Y$ such that
$XY=q^2$, the approximate functional equation for the product
$L$-function gives the formula
\begin{equation}\label{EFA}
L(\demi,\chi^a)L(\demi,\chi^b)=\sum_{m,n\geq
  1}\frac{\chi(m^an^b)}{\sqrt{mn}} V_{a,b,\chi}\Bigl(\frac{mn}{X}\Bigr) +
\frac{\eps(\chi^a)\eps(\chi^b)}{i^{t(\chi^a)+t(\chi^b)}} \sum_{m,n\geq
  1}\frac{\bar{\chi}(m^an^b)}{\sqrt{mn}} V_{a,b,\chi}\Bigl(\frac{mn}{Y}\Bigr)
\end{equation}
where the function $V_{a,b,\chi}$ has rapid decay at infinity and only
depends on~$(\chi^a(-1),\chi^b(-1))$.

To be precise, we apply~\cite[Th.\,5.3]{IK} to the $L$-function
$L(s,\chi^a)L(s,\chi^b)$ with conductor~$q^2$ and gamma factor
$$
\gamma_{a,b}(s)=\pi^{-s} \Gamma\Bigl(\frac{s+t(\chi^a)}{2}\Bigr)
\Gamma\Bigl(\frac{s+t(\chi^b)}{2}\Bigr),
$$
and with the parameter~$X$ in loc. cit. replaced by~$X/\sqrt{q}$.
Furthermore, we choose a test function $G$ such that
\begin{equation}\label{condforG}
  \begin{cases} G
    \text{ is holomorphic and bounded in  the strip} -4 < \Reel(u) < 4,\\
    G \text{ is even,} \\
    G(0)=1.
\end{cases}
\end{equation}

Then the formula~\eqref{EFA} holds with
\begin{equation*} 
  V_{a,b,\chi}(y)=\frac{1}{2i\pi}\int_{(3)} y^{-u}
  G(u)\frac{\gamma_{a,b}(\demi+u)}{\gamma_{a,b}(\demi)}\frac{du}{u}.
\end{equation*}

Note that we can define $V_{a,b,\chi}$ by the same formula even if
either $\chi^a$ or $\chi^b$ is trivial.  For any $A>0$, the following
bounds
\begin{gather}
  \label{eq-r1}
  V_{a,b,\chi}(y)=1+O(y^{A})
  \\
  \label{eq-r2}
  V_{a,b,\chi}(y)\ll y^{-A}
\end{gather}
hold for $y>0$ (see~\cite[Prop.\,5.4]{IK}).

For \emph{any} character $\chi$, including the case where $\chi^a$ or
$\chi^b$ is trivial, we have
$$L(\demi,\chi^a)L(\demi,\chi^b)\ll q^{1/2}$$ (by the convexity bound,
see e.g.~\cite[Th.\,5.23]{IK}).  Moreover, for any~$\eps>0$ and $X$,
$Y\geq 1$, we have
\begin{gather*}
  \sum_{m,n\geq 1}\frac{\chi(m^an^b)}{\sqrt{mn}}
  V_{a,b,\chi}\Bigl(\frac{mn}{X}\Bigr) \ll X^{1/2+\eps}
  \\
  \eps(\chi^a)\eps(\chi^b)\sum_{m,n\geq
    1}\frac{\bar{\chi}(m^an^b)}{\sqrt{mn}}
  V_{a,b,\chi}\Bigl(\frac{mn}{Y}\Bigr) \ll Y^{1/2+\eps}
\end{gather*}
(since $|\eps(\chi)|\leq 1$ for all characters). Thus, bringing back the
contributions of the characters where either $\chi^a$ or $\chi^b$ is
trivial, we have
\begin{multline}\label{eq-first-decomp}
  M_{a,b}(q)=\frac{1}{q-1} \sum_{\chi\mods{q}} \Bigl\{\sum_{m,n\geq
    1}\frac{\chi(m^an^b)}{\sqrt{mn}}
  V_{a,b,\chi}\Bigl(\frac{mn}{X}\Bigr)
  \\
  +
  \frac{\eps(\chi^a)\eps(\chi^b)}{i^{t(\chi^a)+t(\chi^b)}}\sum_{m,n\geq
    1}\frac{\bar{\chi}(m^an^b)}{\sqrt{mn}}
  V_{a,b,\chi}\Bigl(\frac{mn}{Y}\Bigr) \Bigr\}
  \\
  +O(q^{-1/2}+q^{-1}(X^{1/2+\eps}+Y^{1/2+\eps})),
\end{multline}
where the implicit constant depends on~$a$, $b$ and~$\eps$.

We will assume that~$q\geq 3$ and focus on the contribution,
denoted~$M^{\rm even}_{a,b}(q)$, of the even characters to the
sum. Note that we then have $t(\chi^a)=t(\chi^b)=0$, and
$V_{a,b,\chi}$ is independent of~$\chi$, and is denoted simply~$V$. We
have therefore
$$
M^{\rm even}_{a,b}(q)= \frac{1}{q-1}
\sum_{\substack{\chi\mods{q}\\\chi\text{ even}}} \Bigl\{\sum_{m,n\geq
  1}\frac{\chi(m^an^b)}{\sqrt{mn}} V\Bigl(\frac{mn}{X}\Bigr) +
\eps(\chi^a)\eps(\chi^b)\sum_{m,n\geq
  1}\frac{\bar{\chi}(m^an^b)}{\sqrt{mn}} V\Bigl(\frac{mn}{Y}\Bigr)
\Bigr\}.
$$

\begin{remark}
  Note that the sum over even characters can also be interpreted as
  the restriction to the toroidal family which is the image of the
  morphism $\chi\mapsto \chi^2$.
\end{remark}

By orthogonality of Dirichlet characters modulo~$q$, combined
with~(\ref{eq-tab-gauss}) and~(\ref{eq-ab2a2b}), we obtain
\begin{multline}\label{eq-even-terms}
  M^{\rm even}_{a,b}(q)=\frac{1}{q-1}\sum_{\chi\mods
    q}\frac{1}{2}(1+\chi(-1))
  \\
  \Bigl\{\sum_{m,n\geq 1}\frac{\chi(m^an^b)}{\sqrt{mn}}
  V\Bigl(\frac{mn}{X}\Bigr) + \eps(\chi^a)\eps(\chi^b)\sum_{m,n\geq
    1}\frac{\bar{\chi}(m^an^b)}{\sqrt{mn}} V\Bigl(\frac{mn}{Y}\Bigr)
  \Bigr\}
  \\
  =N_{a,b}(X)+P_{a,b}(Y),
\end{multline}
where
$$
N_{a,b}(X)=\frac{1}2\sum_{\substack{m,n\geq 1\\(m^an^b)^2\equiv 1\mods
    q}}\frac{1}{\sqrt{mn}}V\Bigl(\frac{mn}{X}\Bigr)
$$
and
\begin{align*}
  P_{a,b}(Y)&=\frac1{2\sqrt q}\sum_{\substack{m,n\geq
      1\\(mn,q)=1}}\frac{1}{\sqrt{mn}}
(\widetilde{T}_{a,b}(m^an^b;q)+\widetilde{T}_{a,b}(-m^an^b;q))
V\Bigl(\frac{mn}{Y}\Bigr)\\
&=\frac1{2\sqrt q}\sum_{m,n\geq
  1}\frac{1}{\sqrt{mn}}\widetilde{T}_{2a,2b}(m^{2a}n^{2b};q)
V\Bigl(\frac{mn}{Y}\Bigr).
\end{align*}

We see that the toric congruences of Section~\ref{ssec-congruences}
appear in~$N_{a,b}(X)$, and that the exponential sums of
Section~\ref{ssec-exp-th1} appear in~$P_{a,b}(Y)$. The treatment of
these terms now varies depending on whether $a+b=0$ or not, the key
difference being how to handle the first term, and in particular how
to extract the main term.

Before proceeding, we observe that
$$
M_{a,b}(q)=\overline{M_{-a,-b}(q)},
$$
which allows us to assume from now on that $a$ is positive.

\begin{remark}\label{rm-general}
  More generally, let~$A$ be an integral matrix~$A$ with $k$ columns
  as in Sections~\ref{sec-exp-sums} and~\ref{sec-diophantine}, and
  define the subroups $\Hh_A^{\perp}(\Ff_q)$ as
  in~(\ref{eq-perp}). Then the toroidal average
  $$
  \frac{1}{|\Hh_A^{\perp}(\Ff_q)|} \sum_{\uple{\chi}\in \Hh_A^{\perp}(\Ff_q)}
  \prod_{j=1}^kL(\demi,\chi_j)
  $$
  associated to~$A$ will essentially be a combination of expressions of
  the type
  $$
  \sum_{\substack{m_1,\ldots,m_k\geq 1\\ (m_1,\ldots,m_k)\mods{q}\in
      \Hh_A(\Ff_q)}} \frac{1}{\sqrt{m_1\cdots m_k}}
  V_{\uple{\chi}}\Bigl(\frac{m_1\cdots m_k}{X}\Bigr)
  $$
  and
  $$
  q^{\rank(A)/2} \sum_{m_1,\ldots,m_k\geq
    1}\frac{T_{A}(m_1,\ldots,m_k;q)}{\sqrt{m_1\cdots m_k}}
  V_{\uple{\chi}}\Bigl(\frac{m_1\cdots m_k}{Y}\Bigr),
  $$
  for suitable test functions~$V_{\uple{\chi}}$, with $XY=q^k$
  (depending on the parity of the characters, as usual).
\end{remark}

\section{Proof of Theorem~\ref{th-1} for $a+b\not=0$}

\subsection{The main term when $a$, $b$ are positive}

We assume here that $a$ and~$b$ are positive.  The toric congruence
\begin{equation}
  \label{eqcongruence}
  (m^an^b)^2\equiv 1\mods q
\end{equation}
has the solution $(m,n)=(1,1)$, which contributes
$$
\frac{1}{2} V\Bigl(\frac{1}{X}\Bigr)=\frac12+O(X^{-1/2})
$$
to~$N_{a,b}(X)$ (see~(\ref{eq-r1})).  For any other solution~$(m,n)$ of
the congruence with $m$, $n$ positive, we have $mn>1$, hence
$$
\max(m,n)\geq (q-1)^{\frac{1}{a+b}}.
$$

Fix some~$\delta>0$ to be chosen later. Using a dyadic partition of
unity and the rapid decay of~$V$ (see~(\ref{eq-r2})), it follows that
\begin{equation}\label{eq-a0}
  \sum_{\substack{mn>1\\(m^an^b)^2\equiv 1\mods
      q}}\frac{1}{\sqrt{mn}}V\Bigl(\frac{mn}{X}\Bigr)
  \ll (\log q)^2\max_{I,J} N_{2a,2b}(I\times J;q)+O(X^{-1})
\end{equation}
(with notation as in~(\ref{eq-congruences-norm})), where~$(I,J)$ runs
over pairs of intervals of integers $n\leq X$ such that
$$
\max(|I|,|J|)\geq \demi q^{1/(a+b)},\quad\quad |I|\,|J|\ll X^{1+\delta}.
$$

Fix one such pair $(I,J)$. We distinguish two cases.

\textbf{Case 1.} If $\max(|I|,|J|)\geq q/4$, we can use the
bound~(\ref{eq-trivial}) to get
\begin{equation}\label{eq-a1}
  N_{2a,2b}(I\times J;q)\ll
  \frac{(|I||J|)^{1/2}}{q}\ll\frac{X^{1/2+\delta/2}}{q}.
\end{equation}

\par
\textbf{Case 2.} If $\max(|I|,|J|)<q/4$, then we will appeal to
Theorem~\ref{thmpierce}, with~$k=2$. Since~$a/b>0$, we may assume that
$|J|\geq |I|$, by switching the coordinates if needed. We then have
$\sqrt{|I||J|}\leq |J|$.

Under the assumption that
\begin{equation}\label{eq-condition}
  |I|\leq \frac{1}{4} q^{3/4},
\end{equation}
Theorem~\ref{thmpierce} implies the estimate
$$
N_{2a,2b}(I\times J;q)\ll (\log q)^{1/4} |I|^{2/3-1/2}|J|^{1/4-1/2}=(\log
q)^{1/4}|I|^{1/6}|J|^{-1/4}.
$$

Since $|I||J|\leq |J|^2$, we have then
$$
|I|^{1/6}|J|^{-1/4} =(|I||J|)^{1/6}|J|^{-5/12}\leq
(|I||J|)^{1/6-5/24}=(|I||J|)^{-1/24}\leq \max(|I|,|J|)^{-1/24},
$$
and hence 
\begin{equation}\label{eq-a2}
  N_{2a,2b}(I\times J;q)\ll (\log q)^{1/4}  \max(|I|,|J|)^{-1/24}\ll
  (\log q)^{1/4}q^{-1/(24(a+b))}.
\end{equation}
  
From the combination of~(\ref{eq-a0}),~(\ref{eq-a1}) and~(\ref{eq-a2}),
we obtain
$$
\sum_{\substack{mn>1\\(m^an^b)^2\equiv 1\mods
    q}}\frac{1}{\sqrt{mn}}V\Bigl(\frac{mn}{X}\Bigr) \ll (\log q)^2
\Bigl(q^{-1} X^{(1+\delta)/2}+(\log q)^{1/4}q^{-1/(24(a+b))}\Bigr)+X^{-1},
$$
if the condition~(\ref{eq-condition}) is satisfied for all relevant
intervals~$I$ and~$J$.  Since these satisfy
$$
|I|\leq (|I||J|)^{1/2}\ll X^{(1+\delta)/2}
$$
and since $1-\delta<1/(1+\delta)$ if $0<\delta<1$, the
condition~(\ref{eq-condition}) is indeed satisfied, provided $q$ is
large enough and
$$
X\ll q^{3(1-\delta)/2}.
$$

This last estimate also implies that
$X^{(1+\delta)/2}q^{-1}\ll q^{-(1-3\delta)/4}$, and we conclude that
if $\delta$ is fixed and $0<\delta<1$, and if
$X\ll q^{3(1-\delta)/2}$, then we have
\begin{equation}\label{eq-mt1}
  N_{a,b}(X)=\frac{1}{2}+O(q^{-1/(24(a+b))+\eps}+X^{-1})
\end{equation}
for any $\eps>0$, where the implicit constant depends on~$a$, $b$
and~$\eps$.

\subsection{The main term when $a$ is positive and $b$ is negative}

We assume here that $a\geq 1$ and $-b\geq 1$, and recall that
$a+b\not=0$. We denote by~$d\geq 1$ the gcd of~$a$ and~$b$.

In this case the congruence \eqref{eqcongruence} becomes
\begin{equation*}
  m^{2a}\equiv n^{-2b}\mods q.
\end{equation*}

It admits ``systematic'' (positive) solutions of the form
$$
(m,n)=(r^{-b/d},r^{a/d})
$$
for~$r\geq 1$. The contribution~$N^{\rm syst}_{a,b}(X)$ of solutions of this
form to $N_{a,b}(X)$ is equal to
\begin{equation}\label{eq-b0}
  N^{\rm syst}_{a,b}(X)=\frac{1}{2}\sum_{r\geq 1}\frac{1}{r^{(a-b)/2d}}
  V\Bigl(\frac{r^{(a-b)/d}}{X}\Bigr)=\frac{1}{2}\zeta\Bigl(\frac{a-b}{2d}\Bigr)
  +O(X^{-1/2+\eps}),
\end{equation}
for any $\eps>0$, where we note that the fact that $a+b\not=0$ implies
that $(a-b)/(2d)>1$.

We denote by $N^{\rm err}_{a,b}(X)$ the contribution of the positive
solutions $(m,n)$ to the congruence which are not of the form above.
Let $(m,n)$ be one of these. Then the integer
$$
m^{2a}-n^{-2b}=(m^a-n^{-b})(m^a+n^{-b})
$$
is non-zero and divisible by~$q$, so either $m^a$ or $n^{-b}$ is at
least~$q/2$, and hence we obtain the lower bound
\begin{equation}\label{lowerbound2}
  \max(m,n)\geq (\demi q)^{1/\max(a,-b)}.
\end{equation}

Fix again $\delta>0$ small enough. By a dyadic partition of unity and
the rapid decay~(\ref{eq-r2}) of~$V$ (applied for some fixed~$A$ large
enough, depending on~$\delta$), we have
\begin{equation}\label{eq-b1}
  N^{\rm err}_{a,b}(X) \ll (\log q)^2\max_{I,J} N_{2a,2b}(I\times J;q)+X^{-1}
\end{equation}
(with notation as in~(\ref{eq-congruences-norm})), where~$(I,J)$ runs
over pairs of intervals of integers $n\leq X$ such that
$$
\max(|I|,|J|)\gg q^{1/\max(a,-b)},\quad\quad |I|\,|J|\ll X^{1+\delta},
$$
where the implicit constants depend only on~$a$ and~$b$.

Fix one such pair $(I,J)$.  Let
$$
K_+=\max(|I|,|J|),\quad\quad K_-=\min(|I|,|J|).
$$

The bound~(\ref{eq-trivial}) gives the estimate
$$
N_{2a,2b}(I\times J;q)\ll \frac{K_-(K_+/q+1)}{(K_+K_-)^{1/2}}\ll
\frac{X^{1/2+\delta/2}}{q}+\Bigl(\frac{K_-}{K_+}\Bigr)^{1/2}.
$$

We now distinguish two cases again, depending on some parameter~$\eta>0$
to be chosen later, depending on~$a$ and~$b$.

\textbf{Case 1.} If
$$
\frac{K_+}{K_-}\geq q^\eta,
$$
we deduce that
\begin{equation}\label{eq-b2}
  N_{2a,2b}(I\times J;q)\ll \frac{X^{(1+\delta)/2}}{q}+q^{-\eta/2}.
\end{equation}

\textbf{Case 2.} Suppose that $K_+/K_-<q^{\eta}$. The condition
$a+b\not=0$ implies that at least one of $a/b$ or $b/a$ is not an
integer, hence is not a negative integer. Under the condition that
$$
K_+\leq \demi q^{3/4},
$$
we can apply Theorem~\ref{thmpierce} with~$k=2$ in either case with
$(M,N)=(K_+,K_-)$, and deduce that
$$
N_{2a,2b}(I\times J;q)\ll (\log q)^{1/4} K_+^{1/6}K_-^{-1/4}
=(\log q)^{1/4}(K_+K_-)^{1/6}K_-^{-5/12},
$$
and (since $K_-\geq q^{-\eta}\sqrt{K_-K_+}$) further
\begin{align}
  N_{2a,2b}(I\times J;q)\ll (\log q)^{1/4} (K_+K_-)^{-1/24}
  q^{5\eta/12}&=(\log
  q)^{1/4}(|I||J|)^{-1/24}q^{5\eta/12}\notag\\
  &\ll (\log q)^{1/4}q^{-1/(24\max(a,-b))+5\eta/12}.
  \label{eq-b3}
\end{align}

The bounds~(\ref{eq-b2}) and~(\ref{eq-b3}) lead by~(\ref{eq-b1}) to the
estimate
$$
N^{\rm err}_{a,b}(X) \ll (\log q)^2\Bigl(
q^{-1}X^{(1+\delta)/2}+q^{-\eta/2}+(\log q)^{1/4}
q^{-1/(24\max(a,-b))+5\eta/12} \Bigr)+X^{-1},
$$
provided the condition $K_+\leq \demi q^{3/4}$ is satisfied for all
relevant intervals~$I$ and~$J$. Since
$$
K_+^2q^{-\eta}\leq K_+K_-=|I||J|\ll X^{1+\delta},
$$
and $1-\delta<1/(1+\delta)$ if $0<\delta<1$, this condition will be
true for~$q$ large enough as soon as $X\ll q^{(3/2-\eta)(1-\delta)}$.

Since this last condition on~$X$ implies that
$$
q^{-1}X^{(1+\delta)/2}\ll q^{-1/4-\eta/2},
$$
we conclude, by combining this with~(\ref{eq-b0}), that if~$\eta>0$
and $0<\delta<1$, then
\begin{equation}\label{eq-mt2}
  N_{a,b}(X) =\frac{1}{2}\zeta\Bigl(\frac{a-b}{2d}\Bigr)
  +O(X^{-1/2+\eps}+q^{-1/(24\max(a,-b))+5\eta/12+\eps}+q^{-\eta/2})
\end{equation}
for any $\eps>0$, provided $X\ll q^{(3/2-\eta)(1-\delta)}$.

\subsection{Conclusion of the proof}

Since $a+b\not=0$, we can apply Proposition~\ref{pr-tab-weil} to
estimate the exponential sums $\widetilde{T}_{2a,2b}(m^{2a}n^{2b};q)$ in
$P_{a,b}(Y)$. Using the decay of~$V$ (see~(\ref{eq-r2})), we obtain
\begin{equation}\label{Mab2bound}
  P_{a,b}(Y)\ll q^{-1/2}Y^{1/2+\eps}
\end{equation}
for any $\eps>0$, where the implicit constant depends on~$a$, $b$
and~$\eps$.

\begin{remark}\label{rempisa}
  This bound is trivial in the ``balanced'' case, i.e., when $Y=q$. It
  should however be possible with current methods to improve it by a
  factor $q^{-\eta}$ for some $\eta>0$ if $Y$ has size about $q$, but
  this would require much more elaborate arguments, in the spirit of the
  work of Kowalski, Michel and Sawin~\cite[Thm.\,1.3]{Pisa}.
\end{remark}

Combining~(\ref{Mab2bound}) with~(\ref{eq-mt1}) or~(\ref{eq-mt2}), we
obtain in all cases (recalling that $a\geq 1$) the estimate
$$
M^{\rm even}_{a,b}(q)=\demi \alpha(a,b)+O(q^{-1/(24(|a|+|b|))+5\eta/12+\eps}
+X^{-1/2+\eps}+q^{-\eta/2}+q^{-1/2}Y^{1/2+\eps})
$$
where
$$
\alpha(a,b)=
\begin{cases}
  1&\text{ if } ab\geq 1,\\
  \zeta(\tfrac{a-b}{2(a,b)})&\text{ if } ab\leq -1,
\end{cases}
$$
provided $X\ll q^{(3/2-\eta)(1-\delta)}$ with $\eta>0$
and~$0<\delta<1$..

We pick (say) $\delta=1/4$ and $\eta=\frac{c_0}{|a|+|b|}$ for $c_0$
fixed and small enough, and define
$$X=q^{(3/2-\eta)(1-\delta)}=q^{9/8-3\eta/4},$$ so that
$Y=q^{7/8+3\eta/4}$. Inspecting the error terms above, we see that if
$c_0$ is small enough ($c_0<1/20$ suffices), then there exists $c>0$,
depending only on~$c_0$, such that the asymptotic formula
$$
M^{\rm even}_{a,b}(q)=\demi \alpha(a,b)+O(q^{-c/(|a|+|b|))+\eps})
$$
holds for any $\eps>0$, where the implicit constant depends only on~$a$,
$b$ and~$\eps$.

We now consider briefly the contribution from the odd characters. This
is
$$
M^{\rm odd}_{a,b}(q)= \frac{1}{q-1} \sum_{\substack{\chi\mods{q}\\\chi\text{
      odd}}} \Bigl\{\sum_{m,n\geq 1}\frac{\chi(m^an^b)}{\sqrt{mn}}
W\Bigl(\frac{mn}{X}\Bigr) + \frac{\eps(\chi^a)\eps(\chi^b)}{
  i^{\beta}}\sum_{m,n\geq 1}\frac{\bar{\chi}(m^an^b)}{\sqrt{mn}}
W\Bigl(\frac{mn}{Y}\Bigr) \Bigr\}
$$
for some integer~$\beta$ independent of~$\chi$ (depending on the parity
of~$a$ and~$b$) and some function~$W$ that also depends only on the
parity of~$(a,b)$. This leads to the variant
$$
M^{\rm odd}_{a,b}(q)=N'_{a,b}(X)+i^{-\beta}P'_{a,b}(Y)
$$
of~(\ref{eq-even-terms}), where
$$
N'_{a,b}(X)=\frac{1}{2}
\sum_{\substack{m,n\geq 1\\m^an^b\equiv 1\mods
    q}}\frac{1}{\sqrt{mn}}W\Bigl(\frac{mn}{X}\Bigr)
-\frac{1}{2}\sum_{\substack{m,n\geq 1\\m^an^b\equiv -1\mods
    q}}\frac{1}{\sqrt{mn}}W\Bigl(\frac{mn}{X}\Bigr)
$$
and
$$
P'_{a,b}(Y)=\frac1{2\sqrt q}\sum_{\substack{m,n\geq
    1\\(mn,q)=1}}\frac{1}{\sqrt{mn}}
(\widetilde{T}_{a,b}(m^an^b;q)-\widetilde{T}_{a,b}(-m^an^b;q))
W\Bigl(\frac{mn}{Y}\Bigr).
$$

Proceeding exactly as before (with the same choices of~$X$ and~$Y$), we
obtain
$$
M^{\rm odd}_{a,b}(q)=\demi \alpha(a,b)+O(q^{-c/(|a|+|b|)+\eps}).
$$
\par
Finally, combining these with~(\ref{eq-first-decomp}), we conclude the
proof of Theorem~\ref{th-1} in the case where $a+b\not=0$.

\section{Proof of Theorem~\ref{th-1} for $a+b=0$}


We will write simply $M_a(q)=M_{a,-a}(q)$.  As in the previous
section, we treat first in details the average~$M^{\rm even}_a(q)$
over even characters. Starting from~(\ref{eq-even-terms}) with $X=Y=q$
and using the fact that~$\eps(\chi^a)\eps(\chi^{-a})=1$ for~$\chi$
even and $\chi^a$ non-trivial, we get
\begin{equation}
  \label{Maeveninitial}
  M^{\rm even}_{a} (q) =\sum_{m^{2a} \equiv n^{2a} \mods q } \frac
  1{\sqrt{mn}} V \Bigl( \frac {mn}q\Bigr)+O\Bigl(\frac{\log q}{q^{1/2}}\Bigr),
\end{equation}
where $V=V_{0,0,\chi}$ is independent of~$\chi$ even and the second
term accounts for the contribution of the characters $\chi$ such that
$\chi^a=1$.

The equation $m^{2a}= n^{2a}$ has solutions in positive integers given
by $m=n\geq 1$, and no other integral solution. The
contribution~$M^{\rm syst}_a(X)$ of these solutions to $M^{\rm even}_a(q)$ is
equal to
\begin{equation}
  \label{diagonalcontribution}
  M^{\rm syst}_a(X)=\sum_{\substack{m\geq 1\\(m,q)=1}}
  \frac 1{m} V \left( \frac {m^2}q\right)=\frac{1}{2}\log q+C+O(q^{-1/2}),
\end{equation}
for some constant $C$, independent of~$a$ and~$b$
(see~\cite[Lemma\,4.1,\,(4.6)]{IS}, which provides the value
\begin{equation}\label{eq-C}
  C=\gamma+
  \frac{1}{2}\frac{\Gamma'(1/4)}{\Gamma(1/4)}
  -\frac{\log\pi}{2}=
  \frac{\gamma}{2}-\frac{\pi}{4}-\frac{3(\log 2)}{2}-\frac{\log
    \pi}{2}=-2.108876\ldots,
\end{equation}
where~$\gamma$ is Euler's constant).\footnote{\ Actually,
  \cite[Lemma\,4.1]{IS} is established for a test function $G$
  satisfying conditions slightly different from \eqref{condforG} (see
  \cite[p.\.942]{IS}), but the extension to the equality
  \eqref{diagonalcontribution} is straightforward.}

We denote by $M^{\rm err}_{a}(q)$ the contribution of the positive solutions
$(m,n)$ to the congruence which are not of the form above.  Let $(m,n)$
be one of these. Then the integer $$m^{2a}-n^{2a}=(m^a-n^{a})(m^a+n^{a})$$
is non-zero and divisible by~$q$, and we obtain the lower bound
\begin{equation*}
  \max(m,n)\geq (\demi q)^{1/a}
\end{equation*}
(as in~(\ref{lowerbound2})).

We split the sum further as follows:
$$
M^{\rm err}_a(q)=\sum_{\substack{\rho\in\Ff_q^{\times}\\\rho^{2a}=1}}
M^{\rm err}_a(q,\rho),
$$
where
$$
M^{\rm err}_a(q,\rho)=\sum_{\substack{m\equiv \rho n\mods{q}\\
    m^{2a}\not=n^{2a} }} \frac 1{\sqrt{mn}} V \Bigl( \frac
{mn}q\Bigr).
$$

We define
$$
M^*_{1,-1}((\rho,1),I\times J;q) =|\{(m,n)\in I\times J\,\mid\,
n\equiv\rho m\mods{q}\text{ and } m^{2a}\not= n^{2a}\}|.
$$

Let~$\delta>0$ be a fixed parameter to be chosen later. By a dyadic
partition of unity and the rapid decay of~$V$ (\ref{eq-r2}), we have
\begin{equation}\label{eq-e1}
  M^{\rm err}_{a}(q,\rho) \ll (\log q)^2\max_{I,J}
  \frac{M^*_{1,-1}((\rho,1),I\times J;q)}{(|I||J|)^{1/2}}+q^{-1}
\end{equation}
where~$(I,J)$ runs over pairs of intervals of integers $1\leq n\leq q$ such that
$$
\max(|I|,|J|)\gg q^{1/a},\quad\quad |I|\,|J|\ll q^{1+\delta},
$$
the implicit constants depending only on~$a$ and~$b$.

Fix one such pair $(I,J)$. If $M^*_{1,-1}((\rho,1),I\times J;q)$ is
zero, then there is nothing to do. Otherwise, let
$$
K_+=\max(|I|,|J|),\quad\quad K_-=\min(|I|,|J|).
$$

Let $\eta>0$ be a real number to be fixed later.

\textbf{Case 1.} Suppose that
\begin{equation}\label{eq-unbal}
  \frac{K_+}{K_-}\geq q^\eta.
\end{equation}

The bound~(\ref{eq-trivial}) then gives the estimate
\begin{multline}\label{eq-e2}
  \frac{M^*_{1,-1}((\rho,1),I\times J;q)}{(|I||J|)^{1/2}}\ll
  \frac{K_-(K_+/q+1)}{(K_+K_-)^{1/2}}
  \\
  \ll
  \frac{(|I||J|)^{1/2}}{q}+\Bigl(\frac{K_-}{K_+}\Bigr)^{1/2}\ll
  q^{-1/2+\delta/2}+  q^{-\eta/2}.
\end{multline}

\textbf{Case 2.} Suppose that~(\ref{eq-unbal}) is not satisfied, ie.
$$
\frac{K_+}{K_-}< q^\eta.
$$

We have therefore
$$
|I||J|\gg q^{2/a-\eta}.
$$

Suppose first that $\rho^2=1$ and that $\delta+\eta<1$. In view of the
sizes of $m$ and $n$, the congruence $m-\rho n\equiv 0\mods q$ implies
the equality $m-\rho n=0$ if~$q$ is large enough. It follows that
$m^{2a}=n^{2a}$, which shows that
$M^*_{1,-1}((\rho,1),I\times J;q)=0$.

Suppose now that $\rho^2\not=1$; let~$\Lambda_{\rho}\subset \Zz^2$ be
the lattice determined by the condition
$$
m\equiv \rho n\mods{q},
$$
and let~$\lambda_{\rho}\geq 1$ be its minimum. We claim that
$$
\lambda_{\rho}\gg q^{1/a}.
$$

Indeed, let $(m,n)$ be a non-zero solution of the congruence
$m\equiv \rho n\mods{q}$. We then have $m^{2a}\equiv n^{2a}\mods{q}$,
hence either $q$ divides $m^a-n^a$ or $q$ divides $m^a+n^a$. But~$m-n$
and~$m+n$ are non-zero (since~$\rho^2\not=1$), so (as in the proof
of~(\ref{lowerbound2})), we deduce that $\max(|m|,|n|)\gg q^{1/a}$.
By Proposition~\ref{pr-lattice2}, we then have 
$$
M^*_{1,-1}((\rho,1),I\times J;q)\leq M_{1,-1}((\rho,1),I\times J;q)\ll
\frac{|I||J|}{q}+\frac{|I|+|J|}{\lambda_{\rho}}+1,
$$
hence
\begin{align}\nonumber
  \frac{M^*_{1,-1}((\rho,1),I\times J;q)}{(|I||J|)^{1/2}}
  &\ll
    \frac{(|I||J|)^{1/2}}{q}
    +\frac{|I|+|J|}{(|I||J|)^{1/2}}\frac{1}{\lambda_{\rho}}+
    \frac{1}{(|I||J|)^{1/2}}\\
  &\ll q^{-1/2+\delta/2}+q^{\eta/2-1/a}+q^{\eta/2-1/a}.
\label{eq-e4}
\end{align}

From~(\ref{eq-e1}) and the bounds~(\ref{eq-e2}) and~(\ref{eq-e4}), we
deduce that for $\delta+\eta<1$
$$
M^{\rm err}_a(q,\rho)\ll (\log q)^2
(q^{-\eta/2}+q^{-1/2+\delta/2}+q^{\eta/2-1/a}).
$$

Using~\eqref{Maeveninitial} and (\ref{diagonalcontribution}), we conclude that
$$
M^{\rm even}_a(q)=\frac{1}{2}\log q+C+O\Bigl((\log q)^2
(q^{-1/2}+q^{-\eta/2}+q^{-1/2+\delta/2}+q^{\eta/2-1/a})\Bigr).
$$

Taking, e.g., $\delta=1/4$ and~$\eta=1/(2a)$, we
obtain that
$$
M^{\rm even}_a(q)=\frac{1}{2}\log q+C+O(q^{-c/a})
$$
for some absolute constant~$c>0$.
 
Finally, one finds by looking at the functional equation that the
contribution of odd characters is
$$
M^{\rm odd}_a(q) =\sum_{m^{a} \equiv n^{a} \mods q } \frac 1{\sqrt{mn}} V
\Bigl( \frac {mn}q\Bigr) -\sum_{m^{a} \equiv -n^{a} \mods q } \frac
1{\sqrt{mn}} V \Bigl( \frac {mn}q\Bigr),
$$
and hence the same analysis as above gives the asymptotic formula
$$
M^{\rm odd}_a(q)=\frac{1}{2}\log q+C+O(q^{-c/a}),
$$
so that Theorem~\ref{th-1} follows by adding these two terms.

\begin{remark}\label{rm-pol}
  A similar argument leads, mutatis mutandis, to the following result:

  \begin{proposition}
    Let~$f\in \Zz[X]$ be a polynomial of degree~$d\geq 2$, irreducible
    over~$\Qq$. There exists $\delta>0$, depending only on~$d$, such
    that the estimate
    $$
    \frac{1}{q-1} \sum_{\chi\mods{q}} \chi(\rho)|L(\demi,\chi)|^2\ll
    q^{-\delta}
    $$
    holds for all primes~$q$ such that~$f$ has at least one root
    modulo~$q$ and~$\rho$ is an arbitrarily chosen root of~$f\mods{q}$.
  \end{proposition}

  More precisely, the application of the approximate functional equation
  shows that this follows from the bound
  $$
  \sum_{\substack{m,n\geq 1\\ (\rho m)^2\equiv
      n^2\mods{q}}}\frac{1}{\sqrt{mn}}V\Bigl(\frac{mn}{q}\Bigr)\ll
  q^{-\delta},
  $$
  which is estimated exactly like~$M_a^{\rm err}(q,\rho)$ above. The point
  is once again that the minimum of the lattices
  $$
  \{(m,n)\in\Zz^2\,\mid\, \rho m\equiv n\mods{q}\},\quad
  \{(m,n)\in\Zz^2\,\mid\, -\rho m\equiv n\mods{q}\}
  $$
  are large, namely of size $\gg q^{1/d}$.

  It would be interesting to study ``root twists'' of this kind for
  other families, e.g., to bound
  $$
  \sum_{\chi\mods{q}}\chi(\rho)|L(\demi,\chi)|^4, \quad\quad \sum_{f \in
    S_k(q)} \lambda_f(\rho)|L(\demi,f)|^2,
  $$
  where $S_k(q)$ is the family of Hecke eigenforms of conductor~$q$
  with Hecke eigenvalues $\lambda_f(n)$, and~$\rho$ is as above.

\end{remark}

\end{document}